\documentclass[10pt,a4paper]{amsart}
\usepackage{amsmath}
\usepackage{amsfonts}
\usepackage{amsthm}
\usepackage{amssymb}
\usepackage[T1]{fontenc}
\usepackage{graphicx}
\usepackage{hyperref}
\newcommand\C{\mathbb{C}}
\newcommand\CC{\mathcal{C}}
\newcommand\E{\mathbb{E}}
\newcommand\LL{\mathcal{L}}
\newcommand\N{\mathbb{N}}
\renewcommand\P{\mathbb{P}}
\newcommand\Q{\mathbb{Q}}
\newcommand\R{\mathbb{R}}
\newcommand\T{\mathbb{T}}
\newcommand\TT{\mathcal{T}}
\def\e{\varepsilon}
\DeclareMathOperator*{\Res}{Res}
\newtheorem{theorem}{Theorem}[section]
\newtheorem{conjecture}[theorem]{Conjecture}
\newtheorem{corollary}[theorem]{Corollary}
\newtheorem{lemma}[theorem]{Lemma}
\newtheorem{proposition}[theorem]{Proposition}
\newtheorem*{polyas}{P\'{o}lya's Conjecture}
\newtheorem*{0<alpha<1/2}{The $0 < \alpha < 1/2$ Conjecture}
\newtheorem*{alpha=1/2}{The $\alpha = 1/2$ Conjecture}
\theoremstyle{remark}
\newtheorem{remark}[theorem]{Remark}
\theoremstyle{definition}
\newtheorem{definition}[theorem]{Definition}
\begin{document}

\title[The Distribution of $L_{\alpha}(x)$ and P\'{o}lya's Conjecture]{The Distribution of Weighted Sums of the Liouville Function and P\'{o}lya's Conjecture}

\author{Peter Humphries}

\address{Department of Mathematics, Australian National University, Canberra 0200 ACT, Australia}

\email{Peter.Humphries@anu.edu.au}



\thanks{This research was partially supported by an Australian Postgraduate Award.}


\begin{abstract}
Under the assumption of the Riemann hypothesis, the Linear Independence hypothesis, and a bound on negative discrete moments of the Riemann zeta function, we prove the existence of a limiting logarithmic distribution of the normalisation of the weighted sum of the Liouville function, $L_{\alpha}(x) = \sum_{n \leq x}{\lambda(n) / n^{\alpha}}$, for $0 \leq \alpha < 1/2$. Using this, we conditionally show that these weighted sums have a negative bias, but that for each $0 \leq \alpha < 1/2$, the set of all $x \geq 1$ for which $L_{\alpha}(x)$ is positive has positive logarithmic density. For $\alpha = 0$, this gives a conditional proof that the set of counterexamples to P\'{o}lya's conjecture has positive logarithmic density. Finally, when $\alpha = 1/2$, we conditionally prove that $L_{\alpha}(x)$ is negative outside a set of logarithmic density zero, thereby lending support to a conjecture of Mossinghoff and Trudgian that this weighted sum is nonpositive for all $x \geq 17$.
\end{abstract}

\maketitle

\section{Introduction}
\label{sectionPolya}

The Liouville function $\lambda(n)$ is defined as the completely multiplicative function satisfying $\lambda(p) = -1$ for each prime $p$. Thus if $n$ has the prime factorisation $n = p_1^{m_1} \cdots p_r^{m_r}$, where the $p_i$ are primes and the $m_i$ are positive integers, then
	\[\lambda(n) = \lambda\left(p_1^{m_1} \cdots p_r^{m_r}\right) = \lambda(p_1)^{m_1} \times \cdots \times \lambda(p_r)^{m_r} = (-1)^{m_1 + \ldots + m_r}.
\]
That is,
	\[\lambda(n) = \begin{cases}
1 & \text{if $n=1$,}	\\
1 & \text{if $n$ has an even number of prime factors counting multiplicities,}	\\
-1 & \text{if $n$ has an odd number of prime factors counting multiplicities.}
\end{cases}\]
We may study the average behaviour of $\lambda(n)$ via
	\[L(x) = \sum_{n \leq x}{\lambda(n)},
\]
the summatory function of the Liouville function. The behaviour of this summatory function is intimately linked to certain properties of the Riemann zeta function. Recall that the Riemann zeta function $\zeta(s)$ is defined for $\Re(s) > 1$ by the Dirichlet series
	\[\zeta(s) = \sum^{\infty}_{n=1}{\frac{1}{n^s}},
\]
or equivalently by the Euler product
	\[\zeta(s) = \prod_p{\frac{1}{1 - p^{-s}}},
\]
and that $\zeta(s)$ extends meromorphically to the entire complex plane with only a simple pole at $s = 1$ with residue $1$. Now as $\lambda(n)$ is completely multiplicative, the Dirichlet series $\sum^{\infty}_{n=1}{\lambda(n) n^{-s}}$ has the Euler product
	\[\sum^{\infty}_{n=1}{\frac{\lambda(n)}{n^s}} = \prod_p{\frac{1}{1 + p^{-s}}},
\]
for $\Re(s) > 1$, and so by comparing Euler products,
	\[\sum^{\infty}_{n=1}{\frac{\lambda(n)}{n^s}} = \frac{\zeta(2s)}{\zeta(s)}
\]
for $\Re(s) > \sigma_c$, the abscissa of convergence of $\sum^{\infty}_{n=1}{\lambda(n) n^{-s}}$. Via partial summation, we therefore obtain the identity
\begin{equation}\label{lambdaext}
\frac{\zeta(2s)}{\zeta(s)} = s\int^{\infty}_{1}{\frac{L(x)}{x^s} \: \frac{dx}{x}}
\end{equation}
for $\Re(s) > \max\{\sigma_c,0\}$. The identity \eqref{lambdaext} implies that $L(x) = O(x^{\max\{\sigma_c+\e,0\}})$ for every $\e > 0$. As $|\lambda(n)| = 1$ for all $n \in \N$, we certainly know that $\sigma_c \leq 1$. On the other hand, the zeroes of $\zeta(s)$ along the line $\Re(s) = 1/2$ ensure that $\sigma_c \geq 1/2$. Indeed, as $\zeta(2s)$ is holomorphic for $\Re(s) > 1/2$, the zeroes of $\zeta(s)$ in the strip $1/2 < \Re(s) < 1$ determine and are determined by the behaviour of $L(x)$.

\begin{theorem}
The Riemann hypothesis is equivalent to the statement
	\[L(x) = O(x^{1/2+\e})
\]
for every $\e > 0$, where the implied constant is dependent on $\e$.
\end{theorem}

A slightly stronger condition on $L(x)$ can be gleaned through the following result of Landau.

\begin{lemma}[Landau {\cite[Lemma 15.1]{Montgomery}}]{\label{Landau}}
Let $F(x)$ be a real-valued function that is bounded and integrable on any finite interval $[1,X]$, and suppose that there exists $x_0 \in [1,\infty)$ such that $F(x)$ is of constant sign on $[x_0,\infty)$. Let $\sigma_c$ be the infimum of the set of points $\sigma \in \R$ for which $\int^{\infty}_{x_0}{F(x)x^{-\sigma-1} \:dx}$ is finite. Then $s\int^{\infty}_{1}{F(x)x^{-s-1} \: dx}$ is holomorphic in the open half-plane $\Re(s) > \sigma_c$ with a singularity at the point $\sigma_c$.
\end{lemma}

We may take $f(n) = \lambda(n)$ in this theorem and use the fact that $\zeta(\sigma) \neq 0$ for all $\sigma > 0$ in order to obtain the following corollary.

\begin{corollary}
If $L(x)$ is of constant sign for all sufficiently large $x$, then the Riemann hypothesis holds.
\end{corollary}

The former statement and some numerical evidence, namely calculations of $L(x)$ up to $x = 1500$, led to the following conjecture of P\'{o}lya in 1919.

\begin{polyas}[{\cite{Polya}}]
For all $x \geq 2$, we have that
	\[L(x) \leq 0.
\]
\end{polyas}

This conjecture would of course imply the Riemann hypothesis, and in fact something slightly stronger; that all of the zeroes of the Riemann zeta function are simple, as we shall show in Section \ref{sectionConditional}. However, it was soon seen that P\'{o}lya's conjecture leads to an overly restrictive condition on the zeroes of $\zeta(s)$. More precisely, Ingham \cite{Ingham} showed that P\'{o}lya's conjecture implies that the positive imaginary parts of the zeroes of the Riemann zeta function are linearly dependent over the rationals. Such a nontrivial relation seems quite unlikely to be true; indeed, it is conjectured that positive imaginary parts of the zeroes of the Riemann zeta function are linearly independent over the rationals, which we call the Linear Independence hypothesis. It therefore came as no major surprise when Haselgrove \cite{Haselgrove} announced in 1958 a disproof of P\'{o}lya's conjecture. In fact, Tanaka \cite{Tanaka} has since shown that the smallest such value of $x \geq 2$ for which $L(x) > 0$ is $x = 906 \, 150 \, 257$.

Motivated by this problem, Mossinghoff and Trudgian \cite{Mossinghoff} look at generalisations of P\'{o}lya's conjecture by studying weighted sums of the Liouville function.

\begin{definition}
The weighted sum of the Liouville function with weight $\alpha \in \R$ is the summatory function
	\[L_{\alpha}(x) = \sum_{n \leq x}{\frac{\lambda(n)}{n^{\alpha}}}.
\]
\end{definition}

When $\alpha = 0$, this is the summatory function of the Liouville function. Mossinghoff and Trudgian study the possibility of the eventual constancy of sign of $L_{\alpha}(x)$ for sufficiently large $x$, and determine that for certain ranges of $\alpha$, namely $0 \leq \alpha \leq 1$, this is closely related to the Riemann hypothesis. For this range of $\alpha$, the same argument as for \eqref{lambdaext} yields the identity
	\[\frac{\zeta(2(\alpha + s))}{\zeta(\alpha + s)} = s\int^{\infty}_{1}{\frac{L_{\alpha}(x)}{x^s} \: \frac{dx}{x}}
\]
for $\Re(s) > 1 - \alpha$, and Proposition \ref{Landau} then shows that the constancy of sign of $L_{\alpha}(x)$ for sufficiently large $x$ implies the Riemann hypothesis. Mossinghoff and Trudgian then go on to prove that for $1/2 < \alpha < 1$, the Riemann hypothesis implies the eventual constancy of sign of $L_{\alpha}(x)$, and so these two problems are equivalent. For $0 \leq \alpha < 1/2$ and $\alpha = 1$, however, Ingham's argument can be modified to prove that the eventual constancy of sign of $L_{\alpha}(x)$ contradicts the Linear Independence hypothesis, and so for these values of $\alpha$ we would expect $L_{\alpha}(x)$ to change sign infinitely often. This has been proven unconditionally to occur for $\alpha = 0$ and $\alpha = 1$ \cite{Borwein}, but no such sign changes have been found in the range $0 < \alpha < 1/2$. This leads to the following conjecture of Mossinghoff and Trudgian.

\begin{0<alpha<1/2}[Mossinghoff--Trudgian {\cite[Problem 1]{Mossinghoff}}]
For $0 < \alpha < 1/2$, the weighted sum $L_{\alpha}(x)$ changes sign infinitely often.
\end{0<alpha<1/2}

Mossinghoff and Trudgian give a heuristic argument that $L_{\alpha}(x)$ is predominantly negative for $0 < \alpha < 1/2$ \cite[\S 2.1]{Mossinghoff}, and they calculate $L_{1/4}(x)$ up to $x = 10^{12}$ and show that it is always negative for $11 \leq x \leq 10^{12}$ \cite[\S 4]{Mossinghoff}.

Finally, the case $\alpha = 1/2$ is of particular interest. Here Ingham's argument no longer applies, and indeed a heuristic argument of Mossinghoff and Trudgian \cite[\S 2.3]{Mossinghoff} suggests that $L_{1/2}(x)$ truly is eventually of constant sign, which is strengthened by computational evidence showing that $L_{1/2}(x) \leq 0$ for $17 \leq x \leq 10^{12}$ \cite[\S 4]{Mossinghoff}.

\begin{alpha=1/2}[Mossinghoff--Trudgian {\cite[Problem 3]{Mossinghoff}}]
The weighted sum $L_{1/2}(x)$ is nonpositive for all $x \geq 17$.
\end{alpha=1/2}

The aim of this paper is to study these conjectures by following the methods of Rubinstein and Sarnak \cite{Rubinstein} and Ng \cite{Ng} in calculating the logarithmic densities $\delta(P_{\alpha})$ of the sets
	\[P_{\alpha} = \left\{x \in [1,\infty) : L_{\alpha}(x) \leq 0\right\}
\]
for $0 \leq \alpha \leq 1/2$. Recall that for a measurable set $P \in [1,\infty)$, we let
\begin{align*}
\underline{\delta}(P) & = \liminf_{X \to \infty} \frac{1}{\log X} \int\limits_{\left\{x \in [1,X] \cap P\right\}}{\frac{dx}{x}},	\\
\overline{\delta}(P) & = \limsup_{X \to \infty} \frac{1}{\log X} \int\limits_{\left\{x \in [1,X] \cap P\right\}}{\frac{dx}{x}}
\end{align*}
be the lower and upper logarithmic densities respectively of $P$, and if these two limits are equal, we define the logarithmic density $\delta(P)$ of $P$ by
	\[\delta(P) = \underline{\delta}(P) = \overline{\delta}(P).
\]
It is clear that if $\delta(P)$ exists, then $0 \leq \delta(P) \leq 1$, and in particular that if $P$ has finite Lebesgue measure, then $\delta(P) = 0$. It is, of course, more common to consider the natural density
	\[\lim_{Y \to \infty} \frac{1}{Y} \int\limits_{\left\{y \in [1,Y] \cap P\right\}}{dy}.
\]
It is not hard to show that if the natural density of a set $P$ exists, then the logarithmic density of $P$ also exists and is equal to the natural density; we note, however, that the converse is not true.

For the logarithmic densities $\delta(P_{\alpha})$, we prove the following conditional results, which rely on the Riemann hypothesis (which we abbreviate to RH), the Linear Independence hypothesis (which we abbreviate to LI), and a bound
	\[J_{-1}(T) = \sum_{0 < \gamma < T}{|\zeta'(\rho)|^{-2}} \ll T
\]
on negative discrete moments of $\zeta(s)$, which is defined and discussed in Section \ref{sectionmoments}.

\begin{theorem}\label{themaintheorem}
Assume RH, LI, and $J_{-1}(T) \ll T$. Then for $0 \leq \alpha < 1/2$,
	\[1/2 \leq \delta(P_{\alpha}) < 1.
\]
Moreover,
	\[\lim_{\alpha \to 1/2^{-}} \delta(P_{\alpha}) = 1.
\]
\end{theorem}

Thus (conditionally) $L_{\alpha}(x)$ does indeed have a bias towards being negative, and as $\alpha$ tends to $1/2$ from below, this bias becomes stronger. In spite of this, the set of $x \in [1,\infty)$ for which $L_{\alpha}(x)$ is positive has strictly positive logarithmic density. In particular, the set of counterexamples to P\'{o}lya's conjecture has strictly positive logarithmic density.

Theorem \ref{themaintheorem} suggests that $\delta(P_{1/2}) = 1$; that is, that the set of counterexamples to the $\alpha = 1/2$ conjecture has logarithmic density zero. Surprisingly enough, this is somewhat easier to prove than Theorem \ref{themaintheorem}, in the sense that the Linear Independence hypothesis is superfluous.

\begin{theorem}\label{alpha=1/2theorem}
Assume RH and $J_{-1}(T) \ll T$. Then
	\[\delta(P_{1/2}) = 1.
\]
\end{theorem}

Note, however, that this result does not eliminate the possibility that $L_{1/2}(x)$ changes sign infinitely often; it merely states that even though this may occur, $L_{1/2}(x)$ is almost always negative.

The methods of proof for these theorems rely heavily on the techniques developed in \cite{Ng}, where Ng proves related results for $M(x) = \sum_{n \leq x}{\mu(n)}$, the summatory function of the M\"{o}bius function. In turn, Ng's method of constructing a limiting logarithmic distribution for $M(x)/\sqrt{x}$ builds on the seminal work of Rubinstein and Sarnak \cite{Rubinstein} on biases in prime number races.

\section{Unconditional and Conditional Estimates on $L_{\alpha}(x)$}
\label{sectionConditional}

We will now mention some of the known results on the behaviour of $L_{\alpha}(x)$ for $0 \leq \alpha \leq 1/2$. In particular, we discuss upper and lower bounds on $L(x)$, and how often it is positive or negative.

For upper bounds, we immediately note the trivial bound
	\[|L_{\alpha}(x)| \leq \sum_{n \leq x}{n^{-\alpha}} \ll x^{1 - \alpha}.
\]
Of course, this can be strengthened significantly. By elementary means --- that is, without appealing to methods of complex analysis --- it is possible to show that $L_{\alpha}(x) = o(x^{1 - \alpha})$; this is equivalent to the Prime Number Theorem. The strongest unconditional estimates are obtained via analytic methods. These involve determining zero-free regions of $\zeta(s)$ to the left of the line $\Re(s) = 1$. The largest zero-free region that has been proven, independently by Korobov \cite{Korobov} and Vinogradov \cite{Vinogradov}, is the region
	\[\{s + it \in \C : \sigma \geq 1 - c (\log \tau)^{-2/3} (\log \log \tau)^{-1/3}\}
\]
for some effective constant $c > 0$. Using this and known bounds on $1/\zeta(s)$ in this region, we may prove the following bounds via standard methods of contour integration.

\begin{theorem}
For each $0 \leq \alpha \leq 1/2$, there exists an effective constant $c = c(\alpha) > 0$ such that
	\[L_{\alpha}(x) = O\left(x^{1 - \alpha} \exp\left(- c \frac{(\log x)^{3/5}}{(\log \log x)^{1/5}}\right)\right).
\]
\end{theorem}

Next we consider lower bounds for $L_{\alpha}(x)$. It is quite simple to show that $L_{\alpha}(x)$ must indeed be bounded away from zero. We know that
	\[\frac{\zeta(2(\alpha + s))}{\zeta(\alpha + s)} = s\int^{\infty}_{1}{\frac{L_{\alpha}(x)}{x^s} \: \frac{dx}{x}}
\]
for $\Re(s) > \max\{\sigma_c,0\}$, where $\sigma_c$ is the abscissa of convergence of the Dirichlet series for $\zeta(2(\alpha + s))/\zeta(\alpha + s)$. As $\zeta(2(\alpha + s))$ has a simple pole at $s = 1/2 - \alpha$, while $\zeta(\alpha + s)$ is nonvanishing there, we must have that $\sigma_c \geq 1/2 - \alpha$. In particular, the statement $L_{\alpha}(x) = O(x^{1/2 - \alpha - \e})$ cannot be true for any $\e > 0$, as otherwise $\zeta(2(\alpha + s))/\zeta(\alpha + s)$ would be holomorphic in the open half-plane $\Re(s) > 1/2 -\alpha - \e$. This then tells us that at least one of the two statements
	\[\liminf_{x \to \infty} \frac{L_{\alpha}(x)}{x^{1/2 - \alpha - \e}} = -\infty, \qquad \limsup_{x \to \infty} \frac{L_{\alpha}(x)}{x^{1/2 - \alpha - \e}} = \infty
\]
must be true for every $\e > 0$. If it is the latter that is the case, then of course we will have disproved P\'{o}lya's conjecture.

In certain cases, it is possible to prove a slightly stronger result. Namely, for the case $\alpha = 0$, which corresponds to the classical case of the summatory function of the Liouville function, we may prove that
	\[\liminf_{x \to \infty} \frac{L(x)}{\sqrt{x}}, \qquad \limsup_{x \to \infty} \frac{L(x)}{\sqrt{x}}
\]
are both bounded away from zero. This involves relating $L(x)$ to a certain function $I(x)$, defined in \eqref{I(x)}, that we will study later.

\begin{theorem}[Anderson--Stark {\cite[Theorem 1]{Anderson}}]
For any $x_0 > 1$, we have that
	\[\liminf_{x \to \infty} \frac{L(x)}{\sqrt{x}} \leq \frac{L(x_0) - I(x_0)}{\sqrt{x_0}} \leq \limsup_{x \to \infty} \frac{L(x)}{\sqrt{x}}.
\]
\end{theorem}

It is shown unconditionally in  \cite{Anderson} that we have the bound
	\[|I(x) - 1| \leq \frac{2 \sqrt{2} \pi \zeta(3/2)}{\zeta(3) \sqrt{x}},
\]
where
	\[\frac{2 \sqrt{2} \pi \zeta(3/2)}{\zeta(3)} = 6.43700967\ldots.
\]
So if one finds small and large values of $L(x)$ in conjunction with the bound above on $I(x)$, then bounds for $\liminf_{x \to \infty} L(x) /\sqrt{x}$ and $\limsup_{x \to \infty} L(x) /\sqrt{x}$ follow. The former task is more difficult; finding extrema of $L(x)$ turns out to be computationally quite demanding. Nevertheless, work of Borwein, Ferguson, and Mossinghoff \cite{Borwein} has resulted in tabulations of values of $L(x)$ up to and beyond $x = 10^{14}$. Notably, this includes the extremal results
	\[L(72 \, 204 \, 113 \, 780 \, 255) = -11 \, 805 \, 117, \qquad L(351 \, 753 \, 358 \, 289 \, 465) = 1 \, 160 \, 327.
\]
Combining these values with the bounds on $I(x)$ yields the following results.

\begin{theorem}[Borwein--Ferguson--Mossinghoff {\cite[Theorem 2]{Borwein}}]
We have that
\begin{equation}{\label{unconlower}}
\liminf_{x \to \infty} \frac{L(x)}{\sqrt{x}} \leq -1.389278414\ldots, \qquad \limsup_{x \to \infty} \frac{L(x)}{\sqrt{x}} \geq 0.061867262\ldots.
\end{equation}
\end{theorem}

It is worth noting that if extrema for $L_{\alpha}(x)$, $0 < \alpha < 1/2$, are determined, then similar results hold for $L_{\alpha}(x)/x^{1/2 - \alpha}$. For $\alpha = 1/2$, however, we expect that no positive extrema past $x = 17$ occur, and that this extremum does not prove that $\limsup_{x \to \infty} L_{1/2}(x) > 0$ as the related function $I_{1/2}(x)$ is too large at $x = 17$.

We now consider determining upper and lower bounds for $L(x)$ under the assumption of certain unproven conjectures. In Section \ref{sectionPolya}, we noted that the bound
	\[L(x) = O(x^{1/2 + \e})
\]
for every $\e > 0$ is equivalent to the Riemann hypothesis; a modification of this argument shows that the same is true for
	\[L_{\alpha}(x) = O(x^{1/2 - \alpha + \e})
\]
for $0 \leq \alpha \leq 1/2$. We now show that this can be strengthened slightly via work of Soundararajan, and subsequently improved by Balazard and de Roton, on a related problem.

\begin{theorem}[Soundararajan {\cite{Soundararajan}}, Balazard--De Roton {\cite{Balazard}}]
Assume RH. Then for every $\e > 0$, we have the estimate
\begin{equation}{\label{soundmobius}}
M(x) = O\left(\sqrt{x}e^{(\log x)^{1/2} (\log \log x)^{5/2 + \e}}\right),
\end{equation}
where $M(x) = \sum_{n \leq x}{\mu(n)}$ is the summatory function of the M\"{o}bius function.
\end{theorem}

From Soundararajan's bound for $M(x)$, it is simple to determine a similar estimate for $L(x)$.

\begin{corollary}
Assume RH. For each $0 \leq \alpha \leq 1/2$ and for each $\e > 0$, there exists an absolute constant $c = c(\alpha) > 0$ such that
	\[L(x) = O\left(\sqrt{x} e^{c(\log x)^{1/2} (\log \log x)^{5/2 + \e}}\right).
\]
\end{corollary}

\begin{proof}
We recall the identity
	\[\lambda(n) = \sum_{d^2 \mid n}{\mu\left(\frac{n}{d^2}\right)},
\]
and hence that
	\[L(x) = \sum_{d^2 \leq x}{M\left(\frac{x}{d^2}\right)}.
\]
Combining this with \eqref{soundmobius}, we obtain
	\[L(x) \ll \sum_{d \leq \sqrt{x}}{\frac{\sqrt{x}}{d}e^{(\log x/d^2)^{1/2} (\log \log x/d^2)^{5/2 + \e}}} \ll \sqrt{x} e^{(\log x)^{1/2} (\log \log x)^{5/2 + \e}} \sum_{d \leq \sqrt{x}}{\frac{1}{d}},
\]
and this last sum is $\ll \log x$. This yields the result with $c = 1 + \e'$ for any $\e' > 0$.
\end{proof}

For lower bounds, we cannot determine anything new under the assumption of the Riemann hypothesis; indeed, P\'{o}lya's conjecture, namely the false statement that $L(x) \leq 0$ for all $x \geq 2$, implies the Riemann hypothesis. If the Riemann hypothesis is false, on the other hand, we can easily prove strong lower bounds.

\begin{theorem}[cf.\ {\cite[Theorem 15.2]{Montgomery}}]\label{Riemannfalselowerbounds}
Suppose that RH is false, so that $\Theta = \sup\{\Re(\rho) : \zeta(\rho) = 0\} > 1/2$. Then for $0 \leq \alpha \leq 1/2$,
	\[\liminf_{x \to \infty} \frac{L_{\alpha}(x)}{x^{\Theta - \alpha - \e}} = - \infty, \qquad \limsup_{x \to \infty} \frac{L_{\alpha}(x)}{x^{\Theta - \alpha - \e}} = \infty
\]
for every $\e > 0$.
\end{theorem}

In particular, the falsity of the Riemann hypothesis implies the falsity of P\'{o}lya's conjecture and the truth of the $0 < \alpha < 1/2$ conjecture; however, it does imply the falsity of the $\alpha = 1/2$ conjecture. A similar argument can also be used should the Riemann hypothesis hold but $\zeta(s)$ have a zero of order greater than one. Note that this argument only holds for $0 \leq \alpha < 1/2$.

\begin{theorem}[cf.{} {\cite[Theorem 15.3]{Montgomery}}]\label{multzerolowerbounds}
Assume RH and that $\zeta(s)$ has a zero $\rho = 1/2 + i \gamma$ of order $m \geq 2$. Then for $0 \leq \alpha < 1/2$,
	\[\liminf_{x \to \infty} \frac{L_{\alpha}(x)}{x^{1/2 - \alpha} (\log x)^{m - 1}} < 0, \qquad \limsup_{x \to \infty} \frac{L_{\alpha}(x)}{x^{1/2 - \alpha} (\log x)^{m - 1}} > 0.
\]
\end{theorem}

As an immediate corollary, we see that P\'{o}lya's conjecture implies the simplicity of the zeroes of the Riemann zeta function.

For $\alpha = 1/2$, the same method of proof shows that
	\[\liminf_{x \to \infty} \frac{L_{1/2}(x)}{(\log x)^{m - 1}} < 0, \qquad \limsup_{x \to \infty} \frac{L_{1/2}(x)}{(\log x)^{m - 1}} > 0
\]
if $\zeta(s)$ has a zero of order $m \geq 3$. Thus the $\alpha = 1/2$ conjecture can only show that $\zeta(s)$ has zeroes of order at most $2$.

To strengthen the unconditional lower bounds \eqref{unconlower} under the assumption of the Riemann hypothesis and the simplicity of the zeroes of $\zeta(s)$, we must also assume separate hypotheses. As we discussed in Section \ref{sectionPolya}, such lower bounds can be attained assuming the Riemann hypothesis and the Linear Independence hypothesis, which states that for positive imaginary parts $\gamma_1,\ldots,\gamma_n$ of nontrivial zeroes of $\zeta(s)$ and rational constants $c_1,\ldots,c_n$, the equation
	\[c_1 \gamma_1 + \cdots + c_n \gamma_n = 0
\]
has only the solution $c_1 = \ldots = c_n = 0$. A limited scope of numerical computations performed by Bateman et al.\ \cite{Bateman} have failed to determine any linear relations amongst the imaginary parts of nontrivial zeroes of the Riemann zeta function; as it is known that there exist infinitely many zeroes of $\zeta(s)$ along the line $\Re(s) = 1/2$, however, these computations are of little use as evidence for the truth of this conjecture. Theoretical evidence for this conjecture is also somewhat lacking, though conversely there is no known mathematical argument that would suggest the existence of a nontrivial relation connecting the imaginary parts of the zeroes of $\zeta(s)$. Nevertheless, some related results are known; it has been shown that for any $\beta \neq 0$ and for every $\e > 0$, the number of positive integers $k$ lying between $1$ and $T$ such that $\zeta(1/2 + k i \beta) \neq 0$ is at least as large as $T^{5/6 - \e}$ for infinitely many values of $T$ with $T$ tending to infinity \cite[Corollary 9.8]{Lapidus}, which suggests that zeroes of $\zeta(s)$ cannot lie too densely in an arithmetic progression. Note that assuming the Riemann hypothesis, one could replace $5/6$ by $1$. Moreover, Martin and Ng \cite{Martin2}, \cite{Martin3} claim to have extended this result to show that for any $\alpha, \beta \in \R$ with $\beta \neq 0$, there exists a positive constant $c > 0$ such that the number of positive integers $k$ lying between $1$ and $T$ satisfying $\zeta(1/2 + i(\alpha + k \beta)) \neq 0$ is at least as large as $cT/\log T$. Despite these suggestive results, a proof of the Linear Independence hypothesis currently seems very much inaccessible.

The importance of this hypothesis lies in applications of Kronecker's theorem, which states that if $t_1,\ldots,t_n$ are linearly independent over the rationals, then the set
	\[\left\{\left(e^{2\pi i t_1 y}, \ldots, e^{2\pi i t_n y}\right) \in \T^n : y \in \R\right\}
\]
is dense in $\T^n$, where
	\[\T^n = \left\{(z_1,\ldots,z_n) \in \C^n : |z_l| = 1 \text{ for all $1 \leq l \leq n$}\right\}
\]
is the $n$-torus. We will use a variant of this result in Section \ref{sectionlimitingdistribution}, and highlight the connection between the Linear Independence hypothesis and Kronecker's theorem with regards to applications in analytic number theory in Section \ref{sectionFormula}. Firstly, however, we note the following result of Ingham, subsequently extended by Mossinghoff in Trudgian in \cite{Mossinghoff}, which relies crucially on Kronecker's theorem.

\begin{theorem}[Ingham {\cite[Theorem A]{Ingham}}, Mossinghoff--Trudgian {\cite[\S 3]{Mossinghoff}}]\label{Inghamthm}
Assume RH and LI. Then we have that
	\[\liminf_{x \to \infty} \frac{L_{\alpha}(x)}{x^{1/2 - \alpha}} = - \infty, \qquad \limsup_{x \to \infty} \frac{L_{\alpha}(x)}{x^{1/2 - \alpha}} = \infty.
\]
\end{theorem}

Ingham's theorem was proved before a counterexample of P\'{o}lya's conjecture was discovered, and provided theoretical evidence to the falsity of this conjecture. Moreover, the method of proof of Ingham's theorem was instrumental in Haselgrove's disproof of P\'{o}lya's conjecture \cite{Haselgrove}.

\section{Moments of the Riemann Zeta Function}
\label{sectionmoments}

In Section \ref{sectionExpression}, we will determine an explicit expression for $L_{\alpha}(x)$ involving a sum of the form
	\[\sum_{\rho}{\frac{\zeta(2\rho)}{\zeta'(\rho)} \frac{x^{\rho - \alpha}}{\rho - \alpha}},
\]
where the sum is over the nontrivial zeroes $\rho$ of $\zeta(s)$. Assuming the Riemann hypothesis, we have that $|x^{\rho - \alpha}| = x^{1/2 - \alpha}$. It is therefore of importance to know bounds on sums of the form
	\[\sum_{|\gamma| < T}{\frac{|\zeta(2\rho)|}{|\rho - \alpha| |\zeta'(\rho)|}},
\]
where $T$ is large, and the sum is over all zeroes $\rho = 1/2 + i \gamma$ with $|\gamma| < T$. Such bounds, however, turn out to be related to highly challenging open problems. Nevertheless, we can make progress based on knowledge we have on the density of the zeroes of the Riemann zeta function in the strip $0 < \Re(s) < 1$.

\begin{theorem}[{\cite[Corollary 14.3]{Montgomery}}]
For $T \geq 4$, let $N(T)$ denote the number of zeroes $\rho = \beta + i \gamma$ of $\zeta(s)$ in the rectangle $0 < \beta < 1$, $0 < \gamma < T$. Then
\begin{equation}{\label{N(T)}}
N(T) = \frac{1}{2\pi} T \log T - \frac{1}{2\pi}(1 + \log 2\pi) T + O(\log T).
\end{equation}
In particular,
\begin{equation}{\label{N(T+1)}}
N(T+1) - N(T) \ll \log T.
\end{equation}
\end{theorem}

By partial summation, these estimates allow us to determine accurate bounds on sums of the form
	\[\sum_{|\gamma| < T}{\frac{1}{|\rho - \alpha|}}.
\]
Similarly, the size of $\zeta(2\rho)$ can be bounded through classical results of Littlewood on the growth of $\zeta(s)$ on the line $\Re(s) = 1$.

\begin{theorem}[Littlewood {\cite[\S 13.3]{Montgomery}}]
Assume RH. Then for all $|t| \geq 1$,
\begin{align}
|\zeta(1+it)| & \leq 2e^{\gamma_0} \log \log \tau + O(1), {\label{Littlewood}} \\
\frac{1}{|\zeta(1 + it)|} & \leq \frac{12 e^{\gamma_0}}{\pi^2} \log \log \tau + O(1), {\label{Littlewoodlower}}
\end{align}
where $\tau = |t| + 4$ and $\gamma_0$ is the Euler--Mascheroni constant.
\end{theorem}

So it remains to determine bounds on the size of $1/\zeta'(\rho)$. This, however, proves to be highly difficult. Indeed, it is not even known whether all the zeroes of $\zeta(s)$ are simple, so it is conceivable that such a bound may be unattainable. For a precise lower bound, on the other hand, there is a classical result of Littlewood.

\begin{theorem}[Littlewood {\cite[Theorem 13.18, Theorem 13.21]{Montgomery}}]
Assume RH. There exists an absolute constant $c > 0$ such that
\begin{equation}{\label{Littlewoodderivative}}
\zeta'(\rho) \ll \exp\left(c \frac{\log \gamma}{\log \log \gamma}\right)
\end{equation}
for every nontrivial zero $\rho = 1/2 + i\gamma$ of $\zeta(s)$.
\end{theorem}

The following conjecture of Hejhal (for nonnegative powers) and Gonek (for negative powers) gives a reasonably precise rate on the growth of sums of powers of $\zeta'(\rho)$.

\begin{conjecture}[Gonek--Hejhal \cite{Gonek}, \cite{Hejhal}]
Let
	\[J_k(T) = \sum_{0 < \gamma < T}{|\zeta'(\rho)|^{2k}}
\]
be the discrete moment of order $k$ of the Riemann zeta function. Then for all $k \in \R$,
	\[J_k(T) \asymp T (\log T)^{(k+1)^2}.
\]
\end{conjecture}

Note that the Gonek--Hejhal conjecture necessarily implies the simplicity of the zeroes of the Riemann zeta function. This conjecture has been further refined via the work of Hughes, Keating, and O'Connell, who arrive at a more precise form of this conjecture by modelling the Riemann zeta function by the characteristic polynomial of a large random unitary matrix. Their method also suggests that the Gonek--Hejhal conjecture is actually false for $k \leq -3/2$.

\begin{conjecture}[Hughes--Keating--O'Connell \cite{Hughes}]{\label{Hughesconjecture}}
For all $k \in \C$ with $\Re(k) > -3/2$, we have that
\begin{equation}{\label{J_k(T)sym}}
J_k(T) \sim \frac{1}{2\pi} \frac{G^2(k+2)}{G(2k+3)} a(k) T (\log T)^{(k+1)^2},
\end{equation}
where $G(s)$ is the Barnes $G$-function and
	\[a(k) = \prod_{p}{\left(1-\frac{1}{p}\right)^{k^2} \sum^{\infty}_{m=0}{\left(\frac{\Gamma(m+k)}{m!\Gamma(k)}\right)^2 \frac{1}{p^m}}}.
\]
\end{conjecture}

Here the Barnes $G$-function is defined by
	\[G(s+1) = (2\pi)^{s/2} \exp\left(-\frac{1}{2}(s^2 + \gamma_0 s^2 + s)\right) \prod^{\infty}_{n=1}{\left(1+\frac{s}{n}\right)^n \exp\left(-s + \frac{s^2}{2n}\right)},
\]
for $s \in \C$.

Limited progress has been made on these conjectures. In the case of $k$ being a nonnegative integer, the precise asymptotic results are known for $k = 0$ unconditionally, $k = 1$ under the assumption of the Riemann hypothesis, and $k = 2$ up to correct order assuming the Riemann hypothesis. Moreover, recent conditional results have determined the order of $J_k(T)$ up to multiplication by an error term of size $\exp(c\log \log T /\log \log \log T)$.

\begin{theorem}[Milinovich--Ng \cite{Milinovich2}, Milinovich \cite{Milinovich1}]
Assume RH, and suppose that $k$ is a positive integer. Then for each $k$, there is an absolute constant $c > 0$ such that
	\[T (\log T)^{(k+1)^2} \ll J_k(T) \ll T (\log T)^{(k+1)^2} \exp\left(c \frac{\log \log T}{\log \log \log T}\right).
\]
\end{theorem}

In the case of negative powers, which is of most relevance to us, much less progress has been made. Indeed, the only accurate bound is a conditional result on the lower bound on the order of $J_{-1}(T)$.

\begin{theorem}[Gonek \cite{Gonek}]
Assume RH and that all of the zeroes of $\zeta(s)$ are simple. Then
\begin{equation}{\label{J_{-1}(T)lower}}
J_{-1}(T) \gg T.
\end{equation}
\end{theorem}

Observe that this bound is consistent with the asymptotic behaviour predicted in Conjecture \ref{Hughesconjecture}. For $J_{-1/2}(T)$, there is also a lower bound, but Conjecture \ref{Hughesconjecture} suggests that this underestimates the growth of $J_{-1/2}(T)$ by a factor of $(\log T)^{1/4}$.

\begin{theorem}[Garaev--Sankaranarayanan {\cite{Garaev}}]
Unconditionally, we have that
\begin{equation}{\label{J_{-1/2}(T)lower}}
J_{-1/2}(T) \gg T,
\end{equation}
where we restrict the sum to be over only the simple nontrivial zeroes of $\zeta(s)$.
\end{theorem}

For our applications, the lower bound \eqref{J_{-1}(T)lower} is not of any use; it is an upper bound that we require, in the form of the following conjecture.

\begin{conjecture}
We have that
	\[J_{-1}(T) = \sum_{0 < \gamma < T}{\frac{1}{|\zeta'(\rho)|^2}} \ll T.
\]
\end{conjecture}

Note that this implies that the zeroes of the Riemann zeta function are all simple. In conjunction with \eqref{J_{-1}(T)lower}, this conjecture shows that
	\[J_{-1}(T) \asymp T,
\]
while the stronger conjecture \eqref{J_k(T)sym} with $k = - 1/2$ implies that
	\[J_{-1}(T) \sim \frac{3}{\pi^3} T;
\]
this particular asymptotic was first conjectured by Gonek \cite{Gonek}. However, we do not require this stronger estimate in our applications. We also note that the assumption $J_{-1}(T) \ll T$ and the Cauchy--Schwarz inequality, along with the asymptotic \eqref{N(T)}, show that
\begin{equation}{\label{J_{-1/2}(T)upper}}
J_{-1/2}(T) \ll T (\log T)^{1/2}.
\end{equation}
Conjecture \ref{Hughesconjecture} suggests that this overestimates the growth of $J_{-1/2}(T)$ by a factor of $(\log T)^{1/4}$.

The hypothesis $J_{-1}(T) \ll T$ is used by Ng \cite{Ng} in applications concerning the summatory function of the M\"{o}bius function, which involves sums of the form
	\[\sum_{|\gamma| < T}{\frac{1}{|\rho \zeta'(\rho)|}}.
\]
With weighted sums of the Liouville function, however, we require estimates on sums of the form
	\[\sum_{|\gamma| < T}{\frac{|\zeta(2\rho)|}{|\rho - \alpha| |\zeta'(\rho)|}}.
\]
Nevertheless, it is simple to transfer from one bound to the other via the estimate \eqref{Littlewood} of Littlewood.

\begin{lemma}[cf.{} {\cite[Lemma 1]{Ng}}]
Assume RH and $J_{-1}(T) \ll T$, and let $0 \leq \alpha \leq 1/2$. Then
\begin{align}
\sum_{0 < \gamma < T}{\frac{|\zeta(2\rho)|}{|\rho - \alpha| |\zeta'(\rho)|}} & \ll (\log T)^{3/2} \log \log T, {\label{J_{-1/2}(T)}} \\
\sum_{T < \gamma < 2T}{\frac{|\zeta(2\rho)|^2}{(|\rho - \alpha| |\zeta'(\rho)|)^2}} & \ll \frac{(\log \log T)^2}{T}, {\label{J_{-1}(2T) - J_{-1}(T)}} \\
\sum_{0 < \gamma < T}{\frac{|\zeta(2\rho)|}{|\rho - \alpha| |\zeta'(\rho)|}} & \gg \frac{\log T}{\log \log T}, {\label{J_{-1/2}(T)decay}} \\
\sum_{\gamma > T}{\frac{|\zeta(2\rho)|^2}{(|\rho - \alpha| |\zeta'(\rho)|)^2}} & \gg \frac{1}{T (\log \log T)^2}. {\label{J_{-1}(T)decay}}
\end{align}
In all cases, we may choose the implicit constant to be independent of $\alpha$.
\end{lemma}

\begin{proof}
The proof of the first and second estimates are essentially identical to the proof of \cite[Lemma 1]{Ng}: they follow by partial summation and Littlewood's estimate \eqref{Littlewood}, and, for the first estimate, the bound \eqref{J_{-1/2}(T)upper}. For the third estimate, we use \eqref{Littlewoodlower}, partial summation, the fact that the zero of $\zeta(s)$ with least positive imaginary part satisfies $\gamma > 14$, and \eqref{J_{-1/2}(T)lower}, so that
\begin{align*}
\sum_{0 < \gamma < T}{\frac{|\zeta(2\rho)|}{|\rho - \alpha| |\zeta'(\rho)|}} & \gg \frac{1}{\log \log T} \left(\left[\frac{J_{-1/2}(t)}{t}\right]^{T}_{14} + \int^{T}_{14}{\frac{J_{-1/2}(t)}{t^2} \, dt}\right)	\\
& \gg \frac{\log T}{\log \log T}.
\end{align*}
Similarly, for the fourth estimate, \eqref{Littlewoodlower}, partial summation, and \eqref{J_{-1}(T)lower} show that
\begin{align*}
\sum_{\gamma > T}{\frac{|\zeta(2\rho)|^2}{(|\rho - \alpha| |\zeta'(\rho)|)^2}} & \gg \frac{1}{(\log \log T)^2} \left(\left[\frac{J_{-1}(t)}{t^2}\right]^{\infty}_{T} + 2\int^{\infty}_{T}{\frac{J_{-1}(t)}{t^3} \, dt}\right)	\\
& \gg \frac{1}{T (\log \log T)^2}.
\end{align*}
Throughout, the implicit constant may be chosen independently of $\alpha$, as
	\[\gamma^2 \leq |\rho - \alpha|^2 = (1/2 - \alpha)^2 + \gamma^2 \leq 2 \gamma^2.
\qedhere\]
\end{proof}

It may initially seem strange that we assume the conjecture $J_{-1}(T) \ll T$ instead of bounds with $|\zeta(2\rho)|$ involved. This is due to a lack of study of asymptotic bounds of sums over zeroes relating to the latter; while we have the precise conjecture \eqref{J_k(T)sym} of Hughes--Keating--O'Connell, no such conjecture for the sum
	\[K_k(T) = \sum_{0 < \gamma < T}{\frac{|\zeta(2\rho)|^{2k}}{|\zeta'(\rho)|^{2k}}}
\]
exists, though we do note that Ng \cite[\S 8.3]{Ng2} suggests that $K_{-1}(T) \sim T/2\pi$, based on a heuristic method of Gonek. Thus it seems that the presence of $|\zeta(2\rho)|^{-2}$ leads only to a change in the constant in the asymptotic for $K_{-1}(T)$ in comparison to $J_{-1}(T)$. This can be explained by noting that $\zeta(s)$ has essentially a constant mean value along the line $\Re(s) = 1$, in the sense that for any fixed $k > 0$,
	\[\int^{T}_{1}{\left|\zeta(1 + it)\right|^{2k} \, dt} \asymp T;
\]
see \cite{Ivic} for further details.

\section{An Explicit Expression for $L_{\alpha}(x)$}
\label{sectionExpression}

Our goal in this section is to express $L_{\alpha}(x)$ in terms of a sum over the nontrivial zeroes of $\zeta(s)$, and use this explicit expression to create a limiting logarithmic distribution for $L_{\alpha}(x)/x^{1/2 - \alpha}$. We must mention that our method is effective only when limited to the range $0 \leq \alpha < 1/2$; that is, we exclude the case $\alpha = 1/2$. Consequently, our results in Sections \ref{sectionExpression}---\ref{sectionBounds} will focus only on the range $0 \leq \alpha < 1/2$.

\begin{theorem}[Perron's Formula {\cite[Theorem 5.1, Corollary 5.3]{Montgomery}}]
Let $f(n)$ be an arithmetic function whose associated Dirichlet series $\sum^{\infty}_{n=1}{f(n) n^{-s}}$ has abscissa of absolute convergence $\sigma_a \in \R$. If $\sigma_0 > \max\{\sigma_a,0\}$, $x \geq 1$, $T \geq 1$, then the summatory function $F(x) = \sum_{n \leq x}{f(n)}$ is given by
	\[F(x) = \frac{1}{2\pi i} \int^{\sigma_0 + iT}_{\sigma_0 - iT}{ \sum^{\infty}_{n=1}{\frac{f(n)}{n^s}} \frac{x^s}{s} \: ds} + E(x) + R(x,T),
\]
where $E(x) = f(x)/2$ if $x \in \N$ and $0$ otherwise, and
	\[R(x,T) \ll \sum_{\substack{x/2 < n < 2x \\ n \neq x}}{|f(n)|\min\left\{1,\frac{x}{T|x-n|}\right\}} + \frac{x^{\sigma_0}}{T} \sum^{\infty}_{n=1}{\frac{|f(n)|}{n^{\sigma_0}}}.
\]
In particular, $\lim_{T \to \infty} R(x,T) = 0$.
\end{theorem}

\begin{corollary}{\label{PerronTLiou}}
Let $0 \leq \alpha < 1/2$. For $x \geq 1$, $T \geq 1$, and $\sigma_0 = 1 - \alpha + 1/\log x$, we have that
\begin{equation}\label{PerronTLioueq}
L_{\alpha}(x) = \frac{1}{2\pi i} \int^{\sigma_0 + iT}_{\sigma_0 - iT}{\frac{\zeta(2(\alpha + s))}{\zeta(\alpha + s)} \frac{x^s}{s} \: ds} + E_{\alpha}(x) + R_{\alpha}(x,T),
\end{equation}
where
	\[E_{\alpha}(x) = \begin{cases}
\displaystyle \frac{\lambda(x)}{2x^{\alpha}} & \text{if $x \in \N$,}	\\
0 & \text{otherwise,}
\end{cases}\]
and
	\[R_{\alpha}(x,T) \ll \frac{1}{x^{\alpha}} + \frac{x^{1 - \alpha} \log x}{T}.
\]
Moreover, $\lim_{T \to \infty} R_{\alpha}(x,T) = 0$.
\end{corollary}

\begin{proof}
We take $f(n) = \lambda(n)/n^{\alpha}$. To estimate the error term, we note that for $\sigma_0 > 1 - \alpha$,
	\[R_{\alpha}(x,T) \ll \sum_{\substack{x/2 < n < 2x \\ n \neq x}}{\frac{1}{n^{\alpha}} \min\left\{1,\frac{x}{T|x-n|}\right\}} + \frac{x^{\sigma_0} \zeta(\sigma_0 + \alpha)}{T}.
\]
For this sum, we replace the minimum by its first member when $n$ is nearest to $x$, and by its second member for all other $n$, and hence
	\[\sum_{\substack{x/2 < n < 2x \\ n \neq x}}{\frac{1}{n^{\alpha}} \min\left\{1,\frac{x}{T|x-n|}\right\}}\ll \frac{1}{x^{\alpha}} + \frac{x^{1 - \alpha} \log x}{T}.
\]
We then note that $\zeta(\sigma_0) \leq 2/(\sigma_0-1)$ for $\sigma_0 > 1$, and hence for $x \geq 1$, $T \geq 1$,
	\[R_{\alpha}(x,T) \ll \frac{1}{x^{\alpha}} + \frac{x^{1 - \alpha} \log x}{T} + \frac{x^{\sigma_0}}{T(\sigma_0 + \alpha - 1)}.
\]
Choosing $\sigma_0 = 1 - \alpha + 1/\log x$ yields the result.
\end{proof}

The importance of this expression is that we can modify this integral by comparing it with integrals over certain closed curves in the complex plane. By Cauchy's residue theorem, this will allow us to express $L_{\alpha}(x)$ in terms of residues of the integrand in \eqref{PerronTLioueq}. In order to do so, however, we must bound the values of $\zeta(2(\alpha + s))/\zeta(\alpha + s)$ along this curve, for which we require the following results.

\begin{lemma}[{\cite[Corollary 10.5, Theorem 13.18, Theorem 13.23]{Montgomery}}]
Assume RH. Let $s = \sigma + it$ with $|t| \geq 1$. Then for all $\e > 0$, we have that
\begin{equation}{\label{zetauniformbounds}}
\zeta(\sigma + it) \ll \begin{cases}
t^{1/2 - \sigma + \e}	&	\text{if $0 < \sigma < 1/2$,}	\\
t^{\e}	&	\text{if $\sigma \geq 1/2$,}
\end{cases}
\end{equation}
where the implied constant is dependent only on $\e$, and for fixed small $\delta > 0$,
\begin{equation}{\label{1/zetanonuniformbounds}}
\frac{1}{\zeta(\sigma + it)} \ll \begin{cases}
t^{-1/2 + \sigma + \e}	&	\text{if $0 < \sigma \leq 1/2 - \delta$,}	\\
t^{\e}	&	\text{if $\sigma \geq 1/2 + \delta$,}
\end{cases}
\end{equation}
where the implied constant is dependent on $\delta$ and $\e$.
\end{lemma}

\begin{lemma}[{\cite[Theorem 13.23]{Montgomery}}]
Assume RH. There exists a sequence $\TT = \{T_v\}^{\infty}_{v=1}$ with $v \leq T_v \leq v+1$ such that for all $\e > 0$ and $0 < \sigma \leq 2$,
\begin{equation}{\label{1/zetauniformbounds}}
\frac{1}{\zeta(\sigma + iT_v)} \ll T_v^{\e},
\end{equation}
and the implied constant is dependent only on $\e$.
\end{lemma}

By employing these estimates, we are able to determine our explicit expression for $L_{\alpha}(x)$ with an adequately small error term.

\begin{theorem}[cf.{} {\cite{Fawaz}}, {\cite[Equation (4)]{Mossinghoff}}, {\cite[Lemma 4]{Ng}}]{\label{ThmL(x)alongT_v}}
Assume RH and that all of the zeroes of $\zeta(s)$ are simple, and let $0 \leq \alpha < 1/2$. Then for $T_v \in \TT$ and $x \geq 1$, we have that
\begin{equation}{\label{L(x)alongT_v}}
L_{\alpha}(x) = \frac{x^{1/2 - \alpha}}{(1 - 2\alpha) \zeta(1/2)} + \sum_{|\gamma| < T_v}{\frac{\zeta(2\rho)}{\zeta'(\rho)} \frac{x^{\rho - \alpha}}{\rho - \alpha}} + E_{\alpha}(x) + I_{\alpha}(x) + R_{\alpha}(x,T_v),
\end{equation}
where for arbitrary small $0 < \e < 1/2 - \alpha$,
	\[I_{\alpha}(x) = \frac{1}{2\pi i x^{\alpha}} \int^{\e + \alpha + i\infty}_{\e + \alpha - i\infty}{\frac{\zeta(2s)}{\zeta(s)} \frac{x^s}{s - \alpha} \, ds},
\]
and
	\[R_{\alpha}(x,T_v) \ll \frac{1}{x^{\alpha}} + \frac{x^{1 - \alpha} \log x}{T_v} + \frac{x^{1 - \alpha}}{T_v^{1-\e} \log x}
\]
with the implied constant dependent on $\e$ and $\alpha$. Moreover, $\lim_{v \to \infty} R_{\alpha}(x,T_v) = 0$.
\end{theorem}

Note that
	\[\frac{1}{\zeta(1/2)} = -0.6847652\ldots.
\]
It is the presence of the negative term $x^{1/2 - \alpha}/((1 - 2\alpha) \zeta(1/2))$ in \eqref{L(x)alongT_v} that leads to the negative bias of $L_{\alpha}(x)$.

\begin{proof}
By Corollary \ref{PerronTLiou}, we have that
	\[L_{\alpha}(x) = \frac{1}{2\pi i} \int^{\sigma_0 + iT}_{\sigma_0 - iT}{\frac{\zeta(2(\alpha + s))}{\zeta(\alpha + s)} \frac{x^s}{s} \: ds} + E_{\alpha}(x) + O\left(\frac{1}{x^{\alpha}} + \frac{x^{1 - \alpha} \log x}{T}\right)
\]
for $x \geq 1$ and $T \geq 1$, where $\sigma_0 = 1 - \alpha + 1/\log x$. Now
\begin{align*}
\frac{1}{2\pi i} \int^{\sigma_0 + iT}_{\sigma_0 - iT} & {\frac{\zeta(2(\alpha + s))}{\zeta(\alpha + s)} \frac{x^s}{s} \: ds}	\\
& = \frac{1}{2\pi i x^{\alpha}} \oint_{\CC_{\alpha}}{\frac{\zeta(2s)}{\zeta(s)} \frac{x^s}{s - \alpha} \: ds} - \frac{1}{2\pi i x^{\alpha}} \int^{\sigma_1 + \alpha - iT}_{\sigma_1 + \alpha + iT}{\frac{\zeta(2s)}{\zeta(s)} \frac{x^s}{s - \alpha} \: ds}	\\
& \qquad - \frac{1}{2\pi i x^{\alpha}} \left( \int^{\sigma_1 + \alpha + iT}_{\sigma_0 + \alpha + iT} + \int^{\sigma_0 + \alpha - iT}_{\sigma_1 + \alpha - iT}\right) \frac{\zeta(2s)}{\zeta(s)} \frac{x^s}{s - \alpha} \: ds,
\end{align*}
where $0 < \sigma_1 = \e < 1/2 - \alpha$, and $\CC_{\alpha}$ denotes the boundary of the rectangle with vertices $\sigma_0 + \alpha - iT, \sigma_0 + \alpha + iT, \sigma_1 + \alpha + iT, \sigma_1 + \alpha - iT$.

The singularities inside $\CC$ of this integrand occur at $s = 1/2$ and at the nontrivial zeroes of the Riemann zeta function with imaginary part bounded above and below by $T$ and $-T$ respectively. By Cauchy's residue theorem, we therefore have that
	\[\frac{1}{2\pi i x^{\alpha}} \oint_{\CC}{\frac{\zeta(2s)}{\zeta(s)} \frac{x^s}{s - \alpha} \: ds} = \frac{x^{1/2 - \alpha}}{(1 - 2\alpha) \zeta(1/2)} + \sum_{|\gamma| < T}{\frac{\zeta(2\rho)}{\zeta'(\rho)} \frac{x^{\rho - \alpha}}{\rho - \alpha}}.
\]

On the other hand,
	\[-\frac{1}{2\pi i x^{\alpha}} \int^{\sigma_1 + \alpha - iT}_{\sigma_1 + \alpha + iT}{\frac{\zeta(2s)}{\zeta(s)} \frac{x^s}{s - \alpha} \: ds} = I_{\alpha}(x) - \frac{1}{\pi i x^{\alpha}} \Re \left( \int^{\sigma_1 + \alpha + i\infty}_{\sigma_1 + \alpha + iT}{\frac{\zeta(2s)}{\zeta(s)} \frac{x^s}{s - \alpha} \: ds}\right)
\]
and
\begin{multline*}
-\frac{1}{2\pi i x^{\alpha}} \left(\int^{\sigma_1 + \alpha + iT_v}_{\sigma_0 + \alpha + iT_v} + \int^{\sigma_0 + \alpha - iT_v}_{\sigma_1 + \alpha - iT_v}\right){\frac{\zeta(2s)}{\zeta(s)} \frac{x^s}{s - \alpha} \: ds}	\\
= \frac{1}{\pi x^{\alpha}} \Im \left(\int^{\sigma_0 + \alpha + iT_v}_{\sigma_1 + \alpha + iT_v}{\frac{\zeta(2s)}{\zeta(s)} \frac{x^s}{s - \alpha} \: ds}\right).
\end{multline*}
For the former, we use \eqref{zetauniformbounds} and \eqref{1/zetanonuniformbounds} to see that for $0 < \sigma_1 + \alpha < 1/2$ and $|t| \geq 1$,
	\[\zeta(2(\sigma_1 + \alpha + it)) \ll t^{1/2 - 2(\sigma_1 + \alpha) + \e/3}, \qquad \frac{1}{\zeta(\sigma_1 + \alpha + it)} \ll t^{-1/2 + \sigma_1 + \alpha + \e/3},
\]
and hence
\begin{align*}
- \frac{1}{\pi i x^{\alpha}} \Re \left( \int^{\sigma_1 + \alpha + i\infty}_{\sigma_1 + \alpha + iT}{\frac{\zeta(2s)}{\zeta(s)} \frac{x^s}{s - \alpha} \: ds}\right) & \ll x^{\sigma_1} \int^{\infty}_{T}{t^{- 1 - (\sigma_1 + \alpha) + 2\e/3} \, dt}	\\
& \ll \frac{x^{\e}}{T^{\e/3 + \alpha}}.
\end{align*}
For the latter, we have by \eqref{zetauniformbounds} and \eqref{1/zetauniformbounds} that if $s = \sigma + iT$ with $T \in \TT$; that is, $T = T_v$ for some $v$, then for $\sigma_1 + \alpha < \sigma < 1/4$,
	\[\zeta(2(\sigma + iT_v)) \ll T_v^{1/2 - 2\sigma + \e}, \qquad \frac{1}{\zeta(\sigma + iT_v)} \ll T_v^{\e},
\]
whereas for $1/4 \leq \sigma < \sigma_0 + \alpha$,
	\[\zeta(2(\sigma + iT_v)) \ll T_v^{\e/2}, \qquad \frac{1}{\zeta(\sigma + iT_v)} \ll T_v^{\e/2},
\]
and hence
\begin{align*}
\frac{1}{\pi x^{\alpha}} \Im & \left(\int^{\sigma_0 + \alpha + iT_v}_{\sigma_1 + \alpha + iT_v}{\frac{\zeta(2s)}{\zeta(s)} \frac{x^s}{s - \alpha} \: ds}\right)	\\
& \qquad \qquad \ll \frac{T_v^{-1/2+2\e}}{x^{\alpha}} \int^{1/4}_{\sigma_1 + \alpha}{T_v^{-2\sigma} x^{\sigma} \: d\sigma} + \frac{T_v^{-1+\e}}{x^{\alpha}} \int^{\sigma_0 + \alpha}_{1/4}{x^{\sigma} \: d\sigma}	\\
& \qquad \qquad \ll \frac{x^{1/4 - \alpha}}{T_v^{1/2}} + \frac{x^{1 - \alpha}}{T_v^{1-\e} \log x}.
\end{align*}
Combining all these estimates yields the result.
\end{proof}

Taking the limit as $v$ tends to infinity in \eqref{L(x)alongT_v}, so as to eliminate the error term, we obtain a closed-form expression for $L_{\alpha}(x)$.

\begin{corollary}
Assume RH and that all of the zeroes of $\zeta(s)$ are simple, and let $0 \leq \alpha < 1/2$. For all $x \geq 1$, we have that
	\[L_{\alpha}(x) = \frac{x^{1/2 - \alpha}}{(1 - 2\alpha) \zeta(1/2)} + \lim_{v \to \infty} 2 \Re\left(\sum_{0 < \gamma < T_v}{\frac{\zeta(2\rho)}{\zeta'(\rho)} \frac{x^{\rho - \alpha}}{\rho - \alpha}}\right) + E_{\alpha}(x) + I_{\alpha}(x).
\]
\end{corollary}

Here we have used the fact that
	\[\overline{\left(\frac{\zeta(2\rho)}{\zeta'(\rho)} \frac{x^{\rho}}{\rho}\right)} = \frac{\zeta(2\overline{\rho})}{\zeta'(\overline{\rho})} \frac{x^{\overline{\rho}}}{\overline{\rho}},
\]
and that if $\rho$ is a zero of $\zeta(s)$, then so is $\overline{\rho}$.

Though this expression for $L_{\alpha}(x)$ is remarkable in its simplicity, the earlier expression \eqref{L(x)alongT_v} in terms of $x$ and $T$ turns out to be more useful for applications. Unfortunately, this expression is somewhat restrictive, in the sense that we require $T$ to be a member of the sequence $\{T_v\}^{\infty}_{v=1}$. We can remove this restriction by using the hypothesis $J_{-1}(T) \ll T$. Firstly, however, we determine bounds for $I_{\alpha}(x)$.

We define the sequence $\{c(n)\}^{\infty}_{n=1}$ as being the coefficients of the Dirichlet series for $\zeta(2s-1)/\zeta(s)$, and we denote by $C(x), S(x)$ the Fresnel integrals
	\[C(x) = \int^{x}_{0}{\cos\left(\frac{\pi t^2}{2}\right) \: dt}, \qquad S(x) = \int^{x}_{0}{\sin\left(\frac{\pi t^2}{2}\right) \: dt}.
\]
In particular, we have that 
	\[\sum^{\infty}_{n=1}{\frac{|c(n)|}{n^{3/2}}} < \infty, \qquad C(x) = \frac{1}{2} + O\left(\frac{1}{x}\right), \qquad S(x) = \frac{1}{2} + O\left(\frac{1}{x}\right),
\]
where the first estimate holds from the absolute convergence of the Dirichlet series for $\zeta(2s-1)/\zeta(s)$ for $\Re(s) > 1$, while the latter two estimates are proved in \cite[\S 16.56, Equations (3)--(8)]{Watson}.

\begin{lemma}[Fawaz {\cite[Theorem 2]{Fawaz}}]{\label{Fawazlemma}}
Assume RH. Let
	\[I(x) = I_0(x) = \frac{1}{2\pi i} \int^{\sigma_1 + i\infty}_{\sigma_1 - i\infty}{\frac{\zeta(2s)}{\zeta(s)} \frac{x^s}{s} \: ds}
\]
for $0 < \sigma_1 < 1/2$ with $\sigma_1$ independent of $x$. Then
\begin{equation}\label{I(x)}
I(x) = 1 + 2 \sum^{\infty}_{n=1}{\frac{c(n)}{n}(C(\sqrt{nx}) + S(\sqrt{nx}) - 1)}.
\end{equation}
In particular, $I(x)$ is independent of $\sigma_1$, with the bounds
\begin{equation}{\label{Fawazbound}}
I(x) = 1 + O\left(\frac{1}{\sqrt{x}}\right), \qquad I\left(\frac{1}{x}\right) = O\left(\frac{1}{x^{1/2 - \e}}\right)
\end{equation}
for any $0 < \e < 1/2$ as $x$ tends to infinity.
\end{lemma}

The last estimate holds as
	\[\left|I\left(\frac{1}{x}\right)\right| \leq \frac{1}{2\pi x^{\sigma_1}} \int^{\sigma_1 + i\infty}_{\sigma_1 - i\infty}{\frac{|\zeta(2s)|}{|s\zeta(s)|} \: |ds|} \ll \frac{1}{x^{\sigma_1}}
\]
as $x$ tends to infinity, and we may take $\sigma_1 = 1/2 - \e$ for any $0 < \e < 1/2$.

We must also determine bounds for
	\[I_{\alpha}(x) = \frac{1}{2\pi i x^{\alpha}} \int^{\sigma_1 + \alpha + i\infty}_{\sigma_1 + \alpha - i\infty}{\frac{\zeta(2s)}{\zeta(s)} \frac{x^s}{s - \alpha} \, ds}
\]
where $0 < \sigma_1 < 1/2 - \alpha$. We note that
	\[\frac{1}{s - \alpha} = \frac{1}{s} + \frac{\alpha}{s(s - \alpha)}
\]
and hence that
	\[I_{\alpha}(x) = \frac{I(x)}{x^{\alpha}} + \frac{\alpha}{2\pi i x^{\alpha}} \int^{\sigma_1 + \alpha + i\infty}_{\sigma_1 + \alpha - i\infty}{\frac{\zeta(2s)}{\zeta(s)} \frac{x^s}{s(s - \alpha)} \, ds}.
\]
On the other hand, we observe that for any $u > 0$,
	\[\frac{I(u)}{u^{1 + \alpha}} = \frac{1}{2\pi i} \int^{\sigma_1 + \alpha + i\infty}_{\sigma_1 + \alpha - i\infty}{\frac{\zeta(2s)}{\zeta(s)} \frac{u^{s - 1 - \alpha}}{s} \, ds},
\]
and so by integrating over $u$ from $0$ to $x$ and interchanging the order of integration, which is justified by the estimates on $I(u)$ in Lemma \ref{Fawazlemma},
	\[\int^{x}_{0}{\frac{I(u)}{u^{1 + \alpha}} \, du} = \frac{1}{2\pi i} \int^{\sigma_1 + \alpha + i\infty}_{\sigma_1 + \alpha - i\infty}{\frac{\zeta(2s)}{\zeta(s)} \frac{x^{s - \alpha}}{s(s - \alpha)} \, ds}
\]
for all $x > 0$. Thus
	\[I_{\alpha}(x) = \frac{I(x)}{x^{\alpha}} + \alpha \int^{x}_{0}{\frac{I(u)}{u^{1 + \alpha}} \, du}.
\]
In conjunction with the estimates  on $I(u)$ in Lemma \ref{Fawazlemma}, we are therefore able to bound $I_{\alpha}(x)$.

\begin{lemma}[cf.{} {\cite[Theorem 6]{Fawaz}}]{\label{Fawazalphaboundlemma}}
Assume RH, and let $0 \leq \alpha < 1/2$. Then
	\[I_{\alpha}(x) = \alpha \int^{\infty}_{0}{\frac{I(u)}{u^{1 + \alpha}} \, du} +  O\left(\frac{1}{x^{1/2 + \alpha}}\right)
\]
as $x$ tends to infinity.
\end{lemma}

\begin{corollary}[cf.{} {\cite[Lemma 5]{Ng}}]{\label{CorL(x)alongT}}
Assume RH and $J_{-1}(T) \ll T$, and let $0 \leq \alpha < 1/2$. Then for all $T \geq 1$ and $x \geq 1$, we have that
\begin{equation}{\label{L(x)arbitraryT}}
L_{\alpha}(x) = \frac{x^{1/2 - \alpha}}{(1 - 2\alpha) \zeta(1/2)} + \sum_{|\gamma| < T}{\frac{\zeta(2\rho)}{\zeta'(\rho)} \frac{x^{\rho - \alpha}}{\rho - \alpha}} + R_{\alpha}(x,T),
\end{equation}
where for arbitrary small $0 < \e < 1/2 - \alpha$,
\begin{equation}{\label{newerror}}
R_{\alpha}(x,T) \ll 1 + \frac{x^{1 - \alpha} \log x}{T} + \frac{x^{1 - \alpha}}{T^{1-\e} \log x}.
\end{equation}
\end{corollary}

\begin{proof}
By \eqref{L(x)alongT_v}, we have that
	\[L_{\alpha}(x) = \frac{x^{1/2 - \alpha}}{(1 - 2\alpha) \zeta(1/2)} + \sum_{|\gamma| < T_v}{\frac{\zeta(2\rho)}{\zeta'(\rho)} \frac{x^{\rho - \alpha}}{\rho - \alpha}} + E_{\alpha}(x) + I_{\alpha}(x) + R_{\alpha}(x,T_v),
\]
where $\{T_v\}^{\infty}_{v=1}$ is a sequence satisfying $v \leq T_v \leq v+1$. We first deal with $E_{\alpha}(x)$ and $I_{\alpha}(x)$: these are both bounded by a constant. Next, for $T \geq 1$ with $v \leq T_v \leq T \leq v+1$, we have that
	\[L_{\alpha}(x) = \frac{x^{1 - \alpha}}{(1 - 2\alpha) \zeta(1/2)} + \sum_{|\gamma| < T}{\frac{\zeta(2\rho)}{\zeta'(\rho)} \frac{x^{\rho - \alpha}}{\rho - \alpha}} - \sum_{T_v \leq |\gamma| < T}{\frac{\zeta(2\rho)}{\zeta'(\rho)} \frac{x^{\rho - \alpha}}{\rho - \alpha}} + R_{\alpha}(x,T_v)
\]
with $R_{\alpha}(x,T_v)$ now as in \eqref{newerror}. By the Cauchy--Schwarz inequality, \eqref{J_{-1}(2T) - J_{-1}(T)}, and \eqref{N(T+1)}, we have that
\begin{align*}
\left|\sum_{T_v \leq |\gamma| < T}{\frac{\zeta(2\rho)}{\zeta'(\rho)} \frac{x^{\rho - \alpha}}{\rho - \alpha}}\right| & \leq x^{1/2 - \alpha} \left(\sum_{T_v \leq |\gamma| < T}{\frac{|\zeta(2\rho)|^2}{(|\rho - \alpha| |\zeta'(\rho)|)^2}}\right)^{1/2} \left(\sum_{T_v \leq |\gamma| < T}{1}\right)^{1/2}	\\
& \hspace{-1cm} \leq x^{1/2 - \alpha} \left(\sum_{T/2 \leq |\gamma| < T}{\frac{|\zeta(2\rho)|^2}{(|\rho - \alpha| |\zeta'(\rho)|)^2}}\right)^{1/2} (N(T) - N(T-1))^{1/2}	\\
& \hspace{-1cm} \ll \frac{x^{1/2 - \alpha} (\log T)^{1/2} \log \log T}{\sqrt{T}},
\end{align*}
and it remains to note that this error term is dominated by the other error terms.
\end{proof}

\section{The Existence of a Limiting Logarithmic Distribution}
\label{sectionlimitingdistribution}

The goal of the next three sections is to prove the following central theorem concerning $L_{\alpha}(x)/x^{1/2 - \alpha}$, in a similar vein to results of Rubinstein and Sarnak \cite{Rubinstein} and Ng \cite{Ng}.

\begin{theorem}\label{limitingdistribution}
Assume RH, LI, and $J_{-1}(T) \ll T$, and let $0 \leq \alpha < 1/2$. Then the function $L_{\alpha}(x)/x^{1/2 - \alpha}$ has a limiting logarithmic distribution. That is, there exists a probability measure $\nu_{\alpha}$ on $\R$ such that
	\[\lim_{X \to \infty} \frac{1}{\log X} \int\limits_{\left\{x \in [1,X] : L_{\alpha}(x)/x^{1/2 - \alpha} \in B\right\}}{\: \frac{dx}{x}} = \nu_{\alpha}(B)
\]
for every Borel set $B \subset \R$ whose boundary has Lebesgue measure zero. Furthermore, the mean and median of $\nu_{\alpha}$ are equal and given by
	\[\mu_{\alpha} = \frac{1}{(1 - 2\alpha) \zeta(1/2)},
\]
while the variance of $\nu_{\alpha}$ is
	\[\sigma_{\alpha}^2 = 2 \sum_{\gamma > 0}{\frac{|\zeta(2\rho)|^2}{(|\rho - \alpha| |\zeta'(\rho)|)^2}}.
\]
\end{theorem}

Taking $B = (-\infty, 0]$ and using the fact that $1/\zeta(1/2) < 0$ yields the following corollary, which gives a more quantitative description of the negative bias of $L_{\alpha}(x)$.

\begin{corollary}
Assume RH, LI, and $J_{-1}(T) \ll T$, and let $0 \leq \alpha < 1/2$. Then
\begin{equation}{\label{mainequation}}
\frac{1}{2} \leq \delta(P_{\alpha}) \leq 1.
\end{equation}
\end{corollary}

We perhaps ought to justify why we assume the three conjectures in Theorem \ref{limitingdistribution}. It ought to be clear by now that these three conjectures have significant sway on the behaviour of $L_{\alpha}(x)$. Theorem \ref{Riemannfalselowerbounds} shows that, at the very least, violations of the Riemann hypothesis lead to a somewhat quantitative resolution of P\'{o}lya's conjecture. As far as the Linear Independence hypothesis goes, Martin and Ng \cite{Martin1}, working with limiting distributions related to prime number races, claim that such limiting distributions exist under significantly weaker conditions. More precisely, they require only the existence of some $\e > 0$ such that for each sufficiently large $T$, there exist at least $\e T /\log T$ nontrivial zeroes of $\zeta(s)$ with imaginary parts positive and bounded by $T$ such that these imaginary parts satisfy no linear relation over the rationals. As asymptotically there are $(1/2\pi) T \log T$ nontrivial zeroes of $\zeta(s)$ with positive imaginary part bounded above by $T$, this suggests we only require a comparatively sparse set of zeroes of $\zeta(s)$ to have linearly independent imaginary parts. Finally, as for the assumption that $J_{-1}(T) \ll T$, we shall see in Lemma \ref{logmeansquare} its key application towards Theorem \ref{limitingdistribution}. In this case, it certainly seems possible to rely on slightly weaker assumptions, but it is not immediately clear which assumptions would suffice; cf.{} \cite[pp.137--138]{Ng2}.

Our starting point for proving Theorem \ref{limitingdistribution} is the explicit expression \eqref{L(x)arbitraryT} for $L_{\alpha}(x)$. By this and using the fact that $x^{\rho - \alpha} = x^{1/2 - \alpha + i \gamma}$ under the Riemann hypothesis, we observe that for $x \geq 1$ and $1 \leq T < X$ we can write
	\[\frac{L_{\alpha}(x)}{x^{1/2 - \alpha}} = E_{\alpha}^{(T)}(x) + \e_{\alpha}^{(T)}(x),
\]
where
	\[E_{\alpha}^{(T)}(x) = \frac{1}{(1 - 2\alpha) \zeta(1/2)} + \sum_{|\gamma| < T}{\frac{\zeta(2\rho)}{\zeta'(\rho)} \frac{x^{i \gamma}}{\rho - \alpha}}
\]
and
\begin{equation}{\label{error}}
\e_{\alpha}^{(T)}(x) = \sum_{T \leq |\gamma| < X}{\frac{\zeta(2\rho)}{\zeta'(\rho)} \frac{x^{i \gamma}}{\rho - \alpha}} + R_{\alpha}(x,X)
\end{equation}
with
	\[R_{\alpha}(x,X) \ll \frac{1}{x^{1/2 - \alpha}} + \frac{\sqrt{x} \log x}{X} + \frac{\sqrt{x}}{X^{1-\e} \log x}
\]
for arbitrary small $\e > 0$. This decomposition of $L_{\alpha}(x)/x^{1/2 - \alpha}$ into two parts, $E_{\alpha}^{(T)}(x)$ and $\e_{\alpha}^{(T)}(x)$, is crucial in creating a limiting distribution for $L_{\alpha}(x)/x^{1/2 - \alpha}$. Our error part $\e_{\alpha}^{(T)}(x)$ is chosen such that its logarithmic mean square is suitably small; that is, the quantity
	\[\frac{1}{\log X}\int^{X}_{1}{|\e_{\alpha}^{(T)}(x)|^2 \: \frac{dx}{x}}
\]
is uniformly bounded in $X$ and tends to zero as $T$ tends to infinity, as we shall see in Lemma \ref{logmeansquare}.

Our first step in constructing our limiting distribution is the following more general result.

\begin{lemma}[Kronecker--Weyl]{\label{subtorus}}
Let $t_1, \ldots, t_n$ be arbitrary real numbers. Then the topological closure $H$ in the $n$-torus $\T^n$ of the set
	\[\widetilde{H} = \left\{\left(e^{2\pi i t_1 y}, \ldots, e^{2\pi i t_n y}\right) \in \T^n : y \in \R\right\}
\]
is an $r$-dimensional subtorus of $\T^n$, where $0 \leq r \leq n$ is the dimension over $\Q$ of the span of $t_1, \ldots, t_n$. Moreover, for any continuous function $g : \T^n \to \C$, we have that
	\[\lim_{Y \to \infty} \frac{1}{Y} \int^{Y}_{0}{g\left(e^{2\pi i t_1 y}, \ldots, e^{2\pi i t_n y}\right) \: dy} = \int_{H}{g(z) \: d\mu_H(z)},
\]
where $\mu_H$ is the normalised Haar measure on $H$.
\end{lemma}

\begin{lemma}[cf.{} {\cite[Lemma 8]{Ng}}, {\cite[Lemma 2.3]{Rubinstein}}]
Assume RH and $J_{-1}(T) \ll T$, and let $0 \leq \alpha < 1/2$. For each $T \geq 1$, there exists a probability measure $\nu_{\alpha,T}$ on $\R$ such that
	\[\lim_{X \to \infty} \frac{1}{\log X} \int^{X}_{1}{f\left(E_{\alpha}^{(T)}(x)\right) \: \frac{dx}{x}} = \int_{\R}{f(x) \: d\nu_{\alpha,T}(x)}
\]
for all continuous functions $f$ on $\R$.
\end{lemma}

The proof is essentially identical to the proof of \cite[Lemma 8]{Ng}; we include the details for further calculations.

\begin{proof}
We let $N = N(T)$ be the number of zeroes of $\zeta(s)$ with positive imaginary parts at most $T$; we denote these imaginary parts by $\gamma_1, \ldots, \gamma_N$. Then by Lemma \ref{subtorus} with $t_l = \gamma_l/2\pi$, there exists a subtorus $H \subset \T^N$ satisfying
	\[\lim_{Y \to \infty} \frac{1}{Y} \int^{Y}_{0}{g\left(e^{i \gamma_1 y}, \ldots, e^{i \gamma_N y}\right) \: dy} = \int_{H}{g(z) \: d\mu_H(z)}
\]
for every continuous function $g$ on $\T^N$, where $\mu_H$ is the normalised Haar measure on $H$. We now define the probability measure $\nu_{\alpha,T}$ on $\R$ by
	\[\nu_{\alpha,T}(B) = \mu_H(\widetilde{B})
\]
for each Borel set $B \subset \R$, where
	\[\widetilde{B} = \left\{(z_1,\ldots,z_N) \in H : \mu_{\alpha} + 2 \Re \left( \sum^{N}_{l = 1}{b_{\alpha,l} z_l}\right) \in B \right\},
\]
with
\begin{equation}{\label{b_0,b_l}}
\mu_{\alpha} = \frac{1}{(1 - 2\alpha) \zeta(1/2)}, \qquad b_{\alpha,l} = \frac{\zeta(1 + 2i\gamma_l)}{(1/2 - \alpha + i\gamma_l) \zeta'(1/2 + i\gamma_l)}.
\end{equation}
The function $\mu_{\alpha} + 2 \Re \left( \sum^{N}_{l = 1}{b_{\alpha,l} z_l}\right)$ is continuous on $H$, so  $\widetilde{B}$ is a Borel set in $H$, and $\nu_{\alpha,T}$ is a probability measure as $\mu_H$ is the normalised Haar measure on $H$. Now if $f$ is a bounded continuous function on $\R$, we define the function $g(z_1,\ldots,z_N)$ on the $N$-torus $\T^N$ by
	\[g(z_1,\ldots,z_N) = f\left(\mu_{\alpha} + 2 \Re \left( \sum^{N}_{l = 1}{b_{\alpha,l} z_l}\right)\right),
\]
so that $g$ is continuous on $\T^N$ with
	\[f\left(E_{\alpha}^{(T)}\left(e^y\right)\right) = g\left(e^{i \gamma_1 y}, \ldots, e^{i \gamma_N y} \right).
\]
Hence by Lemma \ref{subtorus},
\begin{equation}{\label{g,f}}
\begin{split}
\int_{\R}{f(x) \: d\nu_{\alpha,T}(x)} & = \int_{H}{g(z_1,\ldots,z_N) \: d\mu_H(z_1,\ldots,z_N)}	\\
& = \lim_{Y \to \infty} \frac{1}{Y} \int^{Y}_{0}{g\left(e^{i \gamma_1 y}, \ldots, e^{i \gamma_N y} \right) \: dy}	\\
& = \lim_{Y \to \infty} \frac{1}{Y} \int^{Y}_{0}{f\left(E_{\alpha}^{(T)}\left(e^y\right)\right) \: dy}	\\
& = \lim_{X \to \infty} \frac{1}{\log X} \int^{X}_{1}{f\left(E_{\alpha}^{(T)}(x)\right) \: \frac{dx}{x}},
\end{split}
\end{equation}
as required.
\end{proof}

\begin{theorem}[cf.{} {\cite[Theorem 2]{Ng}}, {\cite[Theorem 1.1]{Rubinstein}}]{\label{nuconstruction}}
Assume RH and $J_{-1}(T) \ll T$, and let $0 \leq \alpha < 1/2$. Then $L_{\alpha}(x)/x^{1/2 - \alpha}$ has a limiting logarithmic distribution $\nu_{\alpha}$ on $\R$. That is, there exists a probability measure $\nu_{\alpha}$ on $\R$ such that
	\[\lim_{X \to \infty} \frac{1}{\log X} \int^{X}_{1}{f \left(\frac{L_{\alpha}(x)}{x^{1/2 - \alpha}}\right) \: \frac{dx}{x}} = \int_{\R}{f(x) \: d\nu_{\alpha}(x)}
\]
for all bounded continuous functions $f$ on $\R$.
\end{theorem}

We omit the details of the proof; it is essentially identical to the proof of \cite[Theorem 2]{Ng}. The major change from this proof is that we require the following estimate on $\e^{(T)}_{\alpha}(x)$.

\begin{lemma}[cf.{} {\cite[Lemma 10]{Ng}}]{\label{logmeansquare}}
Assume RH and $J_{-1}(T) \ll T$, and let $0 \leq \alpha < 1/2$. Then
	\[\lim_{X \to \infty} \frac{1}{\log X} \int^{X}_{1}{|\e_{\alpha}^{(T)}(x)|^2 \: \frac{dx}{x}} \ll \frac{\log T (\log \log T)^2}{T^{1/4}}.
\]
\end{lemma}

The proof of this estimate follows easily from the following bound.

\begin{lemma}[cf.{} {\cite[Lemma 6]{Ng}}]{\label{keyboundlemma}}
Assume RH and $J_{-1}(T) \ll T$, and let $0 \leq \alpha \leq 1/2$. Then for $Z > 0$ and $T < X$,
	\[\int^{eZ}_{Z}{\left|\sum_{T \leq |\gamma| < X}{\frac{\zeta(2\rho)}{\zeta'(\rho)} \frac{x^{i \gamma}}{\rho - \alpha}} \right|^2 \, \frac{dx}{x}} \ll \frac{\log T (\log \log T)^2}{T^{1/4}}.
\]
\end{lemma}

\begin{proof}
The proof is essentially identical to the proof of \cite[Lemma 6]{Ng}, where the estimate
	\[\int^{eZ}_{Z}{\left|\sum_{T \leq |\gamma| < X}{\frac{1}{\zeta'(\rho)} \frac{x^{i \gamma}}{\rho}} \right|^2 \, \frac{dx}{x}} \ll \frac{\log T}{T^{1/4}}
\]
is proved. The replacement of $\rho$ by $\rho - \alpha$ in the denominator does not alter the proof at all, while we may use the estimate \eqref{Littlewood} to control the $\zeta(2\rho)$ term in the numerator, which upon mimicking Ng's proof leads to the additional $(\log \log T)^2$ term on the right-hand side.
\end{proof}

\section{An Explicit Formula for $\widehat{\nu_{\alpha}}$}
\label{sectionFormula}

In Theorem \ref{nuconstruction}, we may use Urysohn's lemma to show that the same result holds with $f$ the characteristic function of a Borel set $B \subset \R$ whose boundary has $\nu_{\alpha}$-measure zero. We would like to take $B = (-\infty,0]$ in order to prove the existence of the logarithmic density of the set $P_{\alpha} = \left\{x \in [1,\infty) : L_{\alpha}(x) \leq 0\right\}$. Unfortunately, we do not know that $\{0\}$ has $\nu_{\alpha}$-measure zero; indeed, we know very little about the properties of $\nu_{\alpha}$. If $\nu_{\alpha}$ is absolutely continuous with respect to the Lebesgue measure --- that is, if $d \nu_{\alpha}(x) = \psi_{\alpha}(x) \: dx$ for some nonnegative Lebesgue-integrable function $\psi_{\alpha}$ --- then $\nu_{\alpha}$ must be atomless, and the existence of this logarithmic density will follow. Without additional assumptions, however, it does not seem possible to prove that this is indeed the case.

Nevertheless, assuming certain additional conjectures allows us to derive an explicit formula for $\widehat{\nu_{\alpha}}$, the Fourier transform of $\nu_{\alpha}$, and with this we may show that $\nu_{\alpha}$ is absolutely continuous with respect to the Lebesgue measure. We assume the Linear Independence hypothesis, namely that the positive imaginary parts of the nontrivial zeroes of the Riemann zeta function are linearly independent over the rational numbers; this allows us to calculate $\widehat{\nu_{\alpha}}$ via the construction of $\nu_{\alpha}$ via Lemma \ref{subtorus} and subsequently show that this Fourier transform is reasonably well-behaved. The ensuing expression for $\widehat{\nu_{\alpha}}$ involves an infinite product of Bessel functions, and so we must first recall certain properties of such functions.

\begin{lemma}
Let
	\[\widetilde{J}_0(x) = \sum^{\infty}_{j = 0}{\frac{(-1)^j}{2^{2j} (j!)^2} x^{2j}}
\]
be the Bessel function of the first kind of order zero, and let
	\[\widetilde{I}_0(x) = \widetilde{J}_0(ix) = \sum^{\infty}_{j = 0}{\frac{x^{2j}}{2^{2j} (j!)^2}}
\]
be the modified Bessel function of the first kind of order zero. Then we have the identities
\begin{align}
\widetilde{J}_0(x) = \int^{1}_{0}{e^{i x \cos (2\pi \theta)} \: d\theta},	\label{Besselcos1}\\
\widetilde{I}_0(x) = \int^{1}_{0}{e^{x \cos (2\pi \theta)} \: d\theta},	\label{Besselcos2}
\end{align}
and the bounds
\begin{align}
|\widetilde{J}_0(x)| & \leq \min\left\{1, \sqrt{\frac{3}{4|x|}}\right\},	\label{Besseldecay1}\\
|\widetilde{I}_0(x)| & \leq e^{x^2/4}	\label{Besseldecay2}
\end{align}
for all $x \in \R$.
\end{lemma}

We note that one would usually write $J_0(x)$ to denote the Bessel function of the first kind of order zero; we instead use the notation $\widetilde{J}_0(x)$ so as to avoid confusion with the discrete moments $J_k(T) = \sum_{0 < \gamma < T}{|\zeta'(\rho)|^{2k}}$ of the Riemann zeta function. Similarly, it is more customary to write $I_0(x)$ to denote the modified Bessel function of order zero, but this coincides with our notation for the integral \eqref{I(x)}.

\begin{proof}
The identity \eqref{Besselcos1} is, after an appropriate change of variables, \cite[\S 3.3 Equation (6)]{Watson}, and then replacing $ix$ by $x$ and using the periodicity of the integrand yields the identity \eqref{Besselcos2}. The triangle inequality thereby implies the bound $|\widetilde{J}_0(x)| \leq 1$. We therefore need only prove the second bound on $|\widetilde{J}_0(x)|$ for $|x| \geq 3/4$. By \cite[\S 7.3 Equation (1), \S 7.31 Equations (1)--(2)]{Watson}, we have that for $x > 0$,
	\[\widetilde{J}_0(x) = \sqrt{\frac{2}{\pi x}} \left(P(x) \cos \left(x - \frac{\pi}{4}\right) + Q(x) \sin \left(x - \frac{\pi}{4}\right)\right)
\]
with
	\[0 < P(x) < 1, \qquad 0 < Q(x) < \frac{1}{8x}.
\]
Noting that $|\cos (x - \pi/4)| \leq 1$ and $|\sin (x - \pi/4)| \leq x/2$ for $x \geq 3/4$, we obtain
	\[|\widetilde{J}_0(x)| \leq \sqrt{\frac{2}{\pi x}} \left(1 + \frac{1}{16}\right) < \sqrt{\frac{3}{4 x}}
\]
for $x \geq 3/4$, and the bound \eqref{Besseldecay1} follows by noting that $\widetilde{J}_0(x)$ is an even function. Finally, the bound \eqref{Besseldecay2} is \cite[\S 3 Equation (2)]{Montgomery2}.
\end{proof}

The following result is akin to the results in \cite[\S 3.1]{Rubinstein} and \cite[Corollary 1]{Ng}.

\begin{theorem}{\label{Besselproof}}
Assume RH, LI, and $J_{-1}(T) \ll T$, and let $0 \leq \alpha < 1/2$. Then the Fourier transform $\widehat{\nu_{\alpha}}(\xi) = \int_{\R}{e^{-i\xi x} \: d\nu_{\alpha}(x)}$ of $\nu_{\alpha}$ is given by
\begin{equation}{\label{nutransform}}
\widehat{\nu_{\alpha}}(\xi) = \exp\left(- \frac{i \xi}{(1 - 2\alpha) \zeta(1/2)}\right) \prod_{\gamma > 0}{\widetilde{J}_0\left( \frac{2|\zeta(2\rho)|\xi}{|\rho - \alpha| |\zeta'(\rho)|} \right)},
\end{equation}
where $\widetilde{J}_0(x)$ is the Bessel function of the first kind of order zero.
\end{theorem}

\begin{proof}
If the positive imaginary parts of the Riemann zeta function are linearly independent over the rational numbers, then by Lemma \ref{subtorus}, for any $N = N(T)$ the topological closure $H$ of the set $\left\{\left(e^{ i \gamma_1 y}, \ldots, e^{i \gamma_N y}\right) \in \T^N : y \in \R\right\}$ is all of $\T^N$. Thus the normalised Haar measure $\mu_H$ on $H$ is the normalised Lebesgue measure $dz_1 \cdots dz_N$, and so by \eqref{g,f} with $\mu_{\alpha},b_{\alpha,l}$, $1 \leq l \leq N$ as in \eqref{b_0,b_l},
	\[\int_{\R}{f(x) \: d\nu_{\alpha,T}(x)} = \int_{\T^N}{f\left(\mu_{\alpha} + 2 \Re \left( \sum^{N}_{l = 1}{b_{\alpha,l} z_l}\right)\right) \: dz_1 \cdots dz_N}
\]
for all bounded continuous functions $f : \R \to \R$. Taking $f(x) = e^{-i\xi x}$, we see that
\begin{align*}
\widehat{\nu_{\alpha,T}}(\xi) & = \int_{\T^N}{\exp\left(-i \mu_{\alpha} \xi - 2 i\xi \Re \left( \sum^{N}_{l = 1}{b_{\alpha,l} z_l}\right) \right) \: dz_1 \cdots dz_N}	\\
& = e^{-i \mu_{\alpha} \xi} \prod^{N}_{l = 1}{\int^{1}_{0}{\exp\left( - 2 i\xi \Re \left(b_{\alpha,l} e^{2\pi i \theta_l}\right) \right) \: d\theta_l}}	\\
& = e^{-i \mu_{\alpha} \xi} \prod^{N}_{l = 1}{\int^{1}_{0}{\exp\left( - 2|b_{\alpha,l}| i\xi \cos(2\pi (\theta_l + \beta_{\alpha,l}))\right) \: d\theta_l}}	\\
& = e^{-i \mu_{\alpha} \xi} \prod^{N}_{l = 1}{\int^{1}_{0}{\exp\left(2|b_{\alpha,l}| i \xi \cos (2\pi \theta_l)\right) \: d\theta_l}}	\\
& = e^{-i \mu_{\alpha} \xi} \prod^{N}_{l = 1}{\widetilde{J}_0(2|b_{\alpha,l}|\xi)}.
\end{align*}
Here we have set $\beta_{\alpha,l} = \arg(b_{\alpha,l})/2\pi$ and used the periodicity of the integrand and the identity \eqref{Besselcos1}. It then follows by the weak convergence of $\nu_{\alpha,T}$ to $\nu_{\alpha}$ that for each $\xi \in \R$,
\begin{align*}
\widehat{\nu_{\alpha}}(\xi) & = \lim_{T \to \infty} \widehat{\nu_{\alpha,T}}(\xi)	\\
& = \lim_{T \to \infty} \exp\left(- \frac{i \xi}{(1 - 2\alpha) \zeta(1/2)}\right) \prod^{N(T)}_{l = 1}{\widetilde{J}_0\left( \frac{2|\zeta(1 + 2i\gamma_l)|\xi}{|1/2 - \alpha + i\gamma_l| |\zeta'(1/2 + i\gamma_l)|} \right)}	\\
& = \exp\left(- \frac{i \xi}{(1 - 2\alpha) \zeta(1/2)}\right) \prod_{\gamma > 0}{\widetilde{J}_0\left( \frac{2|\zeta(2\rho)|\xi}{|\rho - \alpha| |\zeta'(\rho)|} \right)}.
\qedhere\end{align*}
\end{proof}

\begin{corollary}
Assume RH, LI, and $J_{-1}(T) \ll T$, and let $0 \leq \alpha < 1/2$. The mean $\mu_{\alpha}$ and variance $\sigma_{\alpha}^2$ of $\nu_{\alpha}$ are given by
\begin{align*}
\mu_{\alpha} & = \frac{1}{(1 - 2\alpha) \zeta(1/2)},	\\
\sigma_{\alpha}^2 & = 2 \sum_{\gamma > 0}{\frac{|\zeta(2\rho)|^2}{(|\rho - \alpha| |\zeta'(\rho)|)^2}}.
\end{align*}
\end{corollary}

\begin{proof}
As $\widetilde{J}_0(0) = 1$ and $\lim_{l \to \infty} b_{\alpha,l} = 0$, we have that
	\[\log \widetilde{J}_0(2 |b_{\alpha, l}| \xi) = - |b_{\alpha,l}|^2 \xi^2 + O\left(|b_{\alpha,l}|^2 \xi^4\right)
\]
for all sufficiently small $\xi$, uniformly in $l$. Here we have used the Taylor series for $\widetilde{J}_0(x)$ and the fact that $|b_{\alpha,l}|^{2j} \ll |b_{\alpha,l}|^2$ for any $j \geq 1$. Consequently,
	\[\log \widehat{\nu_{\alpha}}(\xi) = i \mu_{\alpha} \xi - \sum^{\infty}_{l = 1}{|b_{\alpha,l}|^2} \xi^2 + O(\xi^4)
\]
in a sufficiently small neighbourhood of the origin; we note that $J_{-1}(T) \ll T$ ensures that the coefficient of $\xi^2$ is finite. We recall that $\log \widehat{\nu_{\alpha}}(\xi)$ is the cumulant-generating function of $\nu_{\alpha}$ with cumulants $\kappa_j$ given by
	\[\log \widehat{\nu_{\alpha}}(\xi) = \sum^{\infty}_{j = 1}{\frac{\kappa_j}{j!} (i\xi)^j}.
\]
As $\kappa_1$ is the mean of $\nu_{\alpha}$ and $\kappa_2$ is the variance, we obtain the result.
\end{proof}

The bound \eqref{Besseldecay1} on $\widetilde{J}_0(x)$ allows us to prove that $\widehat{\nu_{\alpha}}(\xi)$ decays rapidly.

\begin{lemma}[cf.{} {\cite[Lemma 2.1]{Feuerverger}}]{\label{nudecay}}
For all $\e > 0$, there exist positive constants $\beta_1, \beta_2$ such that for all $\xi \in \R$,
\begin{equation}{\label{nuhatdecay}}
|\widehat{\nu_{\alpha}}(\xi)| \leq \beta_1 e^{-\beta_2 |\xi|^{\frac{1}{1+\e}}}.
\end{equation}
\end{lemma}

\begin{proof}
By \eqref{nutransform} and \eqref{Besseldecay1}, we have that
	\[|\widehat{\nu_{\alpha}}(\xi)| \leq \prod_{\gamma > 0}{\min\left\{1, \sqrt{\frac{3|\rho - \alpha| |\zeta'(\rho)|}{8|\zeta(2\rho)||\xi|}}\right\}} \leq \prod^{N(T)}_{l = 1}{\sqrt{\frac{3|\rho - \alpha| |\zeta'(\rho)|}{8|\zeta(2\rho)||\xi|}}}.
\]
for any $T \geq 1$. Now by \eqref{Littlewoodlower} and \eqref{Littlewoodderivative} and the fact that $|\rho - \alpha| \ll \gamma$,
\begin{align*}
\prod^{N(T)}_{l = 1}{\frac{|\rho - \alpha| |\zeta'(\rho)|}{|\zeta(2\rho)|}} & \ll \prod_{0 < \gamma < T}{\exp\left(\log \gamma + c \frac{\log \gamma}{\log \log \gamma} + \log \log \log \gamma\right)}	\\
& \leq \exp\left(N(T) \left(\log T + \widetilde{c} \frac{\log T}{\log \log T}\right)\right)
\end{align*}
for $\widetilde{c} = c + \e$ for any $\e > 0$, while
	\[\prod^{N(T)}_{l = 1}{\frac{3}{8|\xi|}} = \exp\left(N(T)\log \left(\frac{3}{8|\xi|}\right)\right),
\]
and hence
	\[|\widehat{\nu_{\alpha}}(\xi)| \ll \exp\left(\frac{N(T)}{2} \log \left(\frac{3}{8 |\xi|} T \exp\left(\widetilde{c} \frac{\log T}{\log \log T}\right)\right)\right).
\]
Taking $T^{1+\e} = |\xi|$ and recalling by \eqref{N(T)} that $N(T) \ll T \log T$ yields the result.
\end{proof}

We are now in a position to prove Theorem \ref{limitingdistribution}.

\begin{proof}[Proof of Theorem \ref{limitingdistribution}]
The bound \eqref{nuhatdecay} implies that $\widehat{\nu_{\alpha}}$ is a Lebesgue-integrable function, and hence that the inverse Fourier transform
	\[\psi_{\alpha}(x) = \frac{1}{2\pi} \int_{\R}{\widehat{\nu_{\alpha}}(\xi) e^{ix\xi} \: d\xi}
\]
exists and satisfies $\widehat{\psi_{\alpha}} = \widehat{\nu_{\alpha}}$; here $\psi_{\alpha}$ is a continuous Lebesgue-integrable function vanishing at infinity. So $\widehat{\nu_{\alpha}}$ is the Fourier transform of the measure $\psi_{\alpha}(x) \: dx$, and hence by the uniqueness of the Fourier transform, $\nu_{\alpha}(B) = \int_{B}{\psi_{\alpha}(x) \: dx}$ for every Borel set $B \subset \R$. In particular, $\nu_{\alpha}$ is absolutely continuous with respect to the Lebesgue measure on $\R$, and so by Theorem \ref{nuconstruction},
	\[\lim_{X \to \infty} \frac{1}{\log X} \int\limits_{\left\{x \in [1,X] : L_{\alpha}(x)/x^{1/2 - \alpha} \in B\right\}}{\: \frac{dx}{x}} = \nu_{\alpha}(B)
\]
for every Borel set $B \subset \R$ whose boundary has Lebesgue measure zero.

Moreover, via Fourier inversion we have that
\begin{align*}
\psi_{\alpha}\left(\frac{1}{(1 - 2\alpha) \zeta(1/2)} + x\right) & = \frac{1}{2\pi} \int_{\R}{\widehat{\nu}(\xi) \exp\left(i\left(\frac{1}{(1 - 2\alpha) \zeta(1/2)} + x\right)\xi\right) \: d\xi}	\\
& = \frac{1}{2\pi} \int_{\R}{\prod_{\gamma > 0}{\widetilde{J}_0\left(\frac{2|\zeta(2\rho)|\xi}{|\rho - \alpha| |\zeta'(\rho)|} \right)} e^{ix\xi} \: d\xi}.
\end{align*}
Now $\widetilde{J}_0(x)$, and hence $\prod_{\gamma > 0}{\widetilde{J}_0\left(2|\zeta(2\rho)|\xi / |\rho - \alpha| |\zeta'(\rho)| \right)}$, is an even function, and so $\psi_{\alpha}(x)$ must be symmetric about $x = 1/((1 - 2\alpha)\zeta(1/2))$. That is,
	\[\mu_{\alpha} = \frac{1}{(1 - 2\alpha) \zeta(1/2)}
\]
is the median of $\nu_{\alpha}$.
\end{proof}

It is worth noting that we may view the limiting logarithmic distribution $\nu_{\alpha}$ of $L_{\alpha}(x) / x^{1/2 - \alpha}$ from a probabilistic point of view, namely that $\nu_{\alpha}$ is the distribution of a certain random variable $X_{\alpha}$ on $\R$. More precisely, for each positive integer $k$, let $X_k$ be a random variable distributed on the interval $[0,1]$ with the sine distribution, and suppose that the collection $\{X_k\}$ is independent. Let $\mu$ be a constant random variable, and let $\{r_k\}$ be a sequence of positive real numbers satisfying $\sum^{\infty}_{k = 1}{r_k^2} < \infty$. For each positive integer $n$, we then define the random variable
	\[\overline{X}_n = \mu + \sum^{n}_{k = 1}{r_k X_k}.
\]
Then $\overline{X}_n$ converges in distribution to a random variable
	\[X = \mu + \sum^{\infty}_{k = 1}{r_k X_k};
\]
in fact, $\overline{X}_n$ converges almost surely and in mean square to $X$. The fact that each $X_k$ has the sine distribution on $[0,1]$ implies that the characteristic function $\varphi_{X_k}(t)$ of $X_k$ is
	\[\varphi_{X_k}(t) = \E\left(e^{i t X_k}\right) = \widetilde{J}_0(t),
\]
and hence that
	\[\varphi_X(t) = \E\left(e^{i t X}\right) = e^{i \mu t} \prod^{\infty}_{k = 1}{\widetilde{J}_0(r_k t)}.
\]
Equivalently, the Fourier transform $\widehat{\nu_X}(\xi) = \int_{\R}{e^{-i \xi x} \, d\nu_X(x)}$ of the probability measure $\nu_X$ on $\R$ given by $\nu_X(B) = \P(X \in B)$ for each Borel set $B \subset \R$, where $\P$ is the uniform measure on $[0,1]^{\N}$, is given by
	\[\widehat{\nu_X}(\xi) = e^{- i \mu \xi} \prod^{\infty}_{k = 1}{\widetilde{J}_0(r_k \xi)}.
\]
So for $0 \leq \alpha < 1/2$, we may take
\begin{align}
\mu & = \mu_{\alpha} = \frac{1}{(1 - 2\alpha) \zeta(1/2)},	\notag\\
r_k & = r_{\alpha, \gamma} = \frac{2|\zeta(2\rho)|}{|\rho - \alpha| |\zeta'(\rho)|}, \label{rk}
\end{align}
for $\rho = 1/2 + i \gamma$ with $\gamma = \gamma_k > 0$, where we put the positive imaginary parts $\gamma_k$ of the zeroes of $\zeta(s)$ in increasing order. The assumption $J_{-1}(T) \ll T$ ensures that $\sum_{\gamma > 0}{{r_{\alpha, \gamma}}^2} < \infty$. By the uniqueness of the Fourier transform of a measure, we conclude that $\nu_{\alpha}$ is equal to the distribution $\nu_{X_{\alpha}}$ of the random variable
\begin{equation}\label{Xalpha}
X_{\alpha} = \frac{1}{(1 - 2\alpha) \zeta(1/2)} + 2 \sum_{\gamma > 0}{\frac{|\zeta(2\rho)|}{|\rho - \alpha| |\zeta'(\rho)|} X_{\gamma}},
\end{equation}
where each $X_{\gamma}$ is a random variable distributed in $[0,1]$ with the sine distribution, and the collection $\{X_{\gamma}\}$ is independent.

\section{Bounds on $\delta(P_{\alpha})$}
\label{sectionBounds}

Ultimately, one would aim to compute the logarithmic density $\delta(P_{\alpha})$ of $P_{\alpha} = \left\{x \in [1,\infty) : L_{\alpha}(x) \leq 0\right\}$ to some adequate precision; that is, to obtain a precise numerical value of $\nu_{\alpha}((-\infty,0])$ with rigorous error bounds under the hypotheses of Theorem \ref{limitingdistribution}. This has been achieved for measures relating to prime number races by Rubinstein and Sarnak in \cite[\S 4]{Rubinstein}, though adapting their methods for our case would require knowing the exact value of the variance
\begin{equation}{\label{sumoverzeroes}}
\sigma_{\alpha}^2 = 2 \sum_{\gamma > 0}{\frac{|\zeta(2\rho)|^2}{(|\rho - \alpha| |\zeta'(\rho)|)^2}}.
\end{equation}
Unfortunately, this does not seem possible; unconditionally it is not even known if this infinite sum converges, though numerical evidence certainly seems to suggest that this is the case (see \eqref{variance}).

A simpler aspiration than computing $\delta(P_{\alpha})$ explicitly is to find tighter bounds on $\delta(P_{\alpha})$ than those in \eqref{mainequation}. In particular, we would like to prove that strict inequality occurs --- that is, that $1/2 < \delta(P_{\alpha}) < 1$ --- so that we may say that $L_{\alpha}(x)$ is indeed negative more frequently than it is positive, but nevertheless it is positive a significant portion of the time (in the sense of logarithmic density). One method of proving these estimates would be to show that
\begin{equation}{\label{desiredinequalities}}
0 < \nu_{\alpha}\left(\left(\mu_{\alpha},0\right)\right) < \frac{1}{2}.
\end{equation}
Following the methods of Feuerverger and Martin \cite[Lemma 2.1]{Feuerverger}, if we were able to show that $|\widehat{\nu_{\alpha}}(\xi)| \leq \beta_1 e^{-\beta_2 |\xi|}$ for some $\beta_1, \beta_2 > 0$ and for all $\xi \in \R$, then a Paley--Wiener-type theorem would allow us to conclude that the probability density function $\psi_{\alpha}$ of $\nu_{\alpha}$ extends to a holomorphic function in the strip $\{z \in \C : |\Im(z)| < \beta_2\}$. As the zeroes of nonzero holomorphic functions cannot have an accumulation point, $\psi$ must be nonvanishing on open sets of $\R$, which would yield the desired estimates. Unfortunately, the bound on $\widehat{\nu_{\alpha}}(\xi)$ obtained in Lemma \ref{nudecay} is insufficient to conclude this, and the methods used for proving this lemma do not seem to be able to yield the required sharper estimate.

Another approach towards yielding information about $\delta(P_{\alpha})$ is via bounds for the tails of $\nu_{\alpha}$. In \cite[Corollary 12]{Ng}, Ng proves bounds of the form
	\[\exp(-\exp(c_1 V^{4/5})) \ll \nu([V,\infty)) \ll \exp(-\exp(c_2 V^{4/5}))
\]
for some absolute constants $c_1, c_2 > 0$. Here $\nu$ is the limiting logarithmic distribution of $M(x) = \sum_{n \leq x}{\mu(n)}$, the summatory function of the M\"{o}bius function, and the proof is conditional on the Riemann hypothesis, the Linear Independence hypothesis, and that
	\[\sum_{0 < \gamma < T}{\frac{1}{|\rho \zeta'(\rho)|}} \asymp (\log T)^{5/4} \qquad \text{and} \qquad \sum_{\gamma > T}{\frac{1}{|\rho \zeta'(\rho)|^2}} \asymp \frac{1}{T}.
\]
A similar, albeit somewhat weaker, result holds with regards to the limiting logarithmic distribution of $L_{\alpha}(x)/x^{1/2 - \alpha}$. This relies on the following result of Montgomery.

\begin{proposition}[Montgomery {\cite{Montgomery2}}]\label{probmom}
Let $\{X_k\}$ be an independent collection of random variables with the sine distribution on $[0,1]$. Let $\mu$ be a constant random variable, and let $\{r_k\}$ be a sequence of positive real numbers satisfying $\sum^{\infty}_{k = 1}{r_k^2} < \infty$. Then the random variable
	\[X = \mu + \sum^{\infty}_{k = 1}{r_k X_k}
\]
has moment-generating function
	\[M_X(t) = \E\left(e^{t X}\right) = e^{i \mu t} \prod^{\infty}_{k = 1}{\widetilde{I}_0(r_k t)},
\]
where $\widetilde{I}_0(x)$ is the modified Bessel function of the first kind of order zero. Furthermore, we have the following bound on the tail of the distribution of $X$:
	\[\P\left(X \geq \mu + 2 \sum^{n}_{k = 1}{r_k}\right)
\geq 2^{-40} \exp\left(- 100 \left(\sum^{n}_{k = 1}{r_k}\right)^2 \left(\sum^{\infty}_{k = n + 1}{r_k^2}\right)^{-1}\right).
\]
\end{proposition}

\begin{corollary}[cf.{} {\cite[\S 4.2]{Ng}}]
Assume RH, LI, and $J_{-1}(T) \ll T$, and let $0 \leq \alpha < 1/2$. Then for every $\e > 0$, there exists an absolute constant $c > 0$ independent of $\alpha$ such that all $V \geq \mu_{\alpha}$,
\begin{equation}{\label{lowertail}}
\nu_{\alpha}([V,\infty)) \geq \exp\left(-\exp\left(c (V - \mu_{\alpha})^{1 + \e}\right)\right).
\end{equation}
\end{corollary}

\begin{proof}
From Proposition \ref{probmom}, we have that for any $T \geq 1$,
	\[\P\left(X_{\alpha} \geq \mu_{\alpha} + 2 \sum_{0 < \gamma < T}{r_{\alpha,\gamma}}\right)
\geq 2^{-40} \exp\left(- 100 \left(\sum_{0 < \gamma < T}{r_{\alpha,\gamma}}\right)^2 \left(\sum_{\gamma > T}{{r_{\alpha,\gamma}}^2}\right)^{-1}\right),
\]
with $X_{\alpha}$ as in \eqref{Xalpha} and $r_{\alpha,\gamma}$ as in \eqref{rk}. Equivalently,
\begin{multline*}
\nu_{\alpha}\left(\left[\mu_{\alpha} + \sum_{0 < \gamma < T}{\frac{|\zeta(2 \rho)|}{|\rho - \alpha| |\zeta'(\rho)|}}, \infty\right)\right)	\\
\geq 2^{-40} \exp\left(- 50 \left(2 \sum_{0 < \gamma < T}{\frac{|\zeta(2 \rho)|}{|\rho - \alpha| |\zeta'(\rho)|}}\right)^2 \left(2 \sum_{\gamma > T}{\frac{|\zeta(2 \rho)|^2}{(|\rho - \alpha| |\zeta'(\rho)|)^2}}\right)^{-1}\right).
\end{multline*}
Now by \eqref{J_{-1/2}(T)} and \eqref{J_{-1}(T)decay},
	\[\left(2 \sum_{0 < \gamma < T}{\frac{|\zeta(2 \rho)|}{|\rho - \alpha| |\zeta'(\rho)|}}\right)^2 \left(2 \sum_{\gamma > T}{\frac{|\zeta(2 \rho)|^2}{(|\rho - \alpha| |\zeta'(\rho)|)^2}}\right)^{-1} \ll T (\log T)^3 (\log \log T)^4,
\]
with the implied constant independent of $\alpha$. On the other hand, \eqref{J_{-1/2}(T)decay} implies that there exists some absolute constant $c > 0$ independent of $\alpha$ such that for sufficiently large $T$,
	\[\sum_{0 < \gamma < T}{\frac{|\zeta(2 \rho)|}{|\rho - \alpha| |\zeta'(\rho)|}} \geq c \frac{\log T}{\log \log T}.
\]
We let $V = c \log T / \log \log T$, so that for any $\e > 0$ there exists an absolute constant $\widetilde{c} > 0$ independent of $\alpha$ such that
	\[50 \left(2 \sum_{0 < \gamma < T}{\frac{|\zeta(2 \rho)|}{|\rho - \alpha| |\zeta'(\rho)|}}\right)^2 \left(2 \sum_{\gamma > T}{\frac{|\zeta(2 \rho)|^2}{(|\rho - \alpha| |\zeta'(\rho)|)^2}}\right)^{-1} \leq \exp\left(\widetilde{c} V^{1 + \e}\right).
\]
Choosing $\widetilde{c}$ sufficiently large, we conclude that for any $V \geq 0$,
	\[\nu_{\alpha}\left(\left[V + \mu_{\alpha}, \infty\right)\right) \geq \exp\left(- \exp\left(\widetilde{c} V^{1 + \e}\right)\right).
\qedhere\]
\end{proof}

\begin{remark}
Montgomery's theorem \cite[\S 3 Theorem 1]{Montgomery2} as stated in Proposition \ref{probmom} actually requires that the sequence $\{r_k\}$ be monotonically nonincreasing with $k$ and tending to zero; with $r_k = r_{\alpha,\gamma}$, the latter is certainly true, but the former is not. Nevertheless, this can be corrected by reordering these terms accordingly. Indeed, the crucial point is having the correct order of magnitude in the bounds for $\sum_{0 < \gamma < T}{r_{\alpha,\gamma}}$ and $\sum_{0 < \gamma < T}{{r_{\alpha,\gamma}}^2}$, and reordering the terms by size certainly ensures this is the case.
\end{remark}

As the absolute constant $c$ in \eqref{lowertail} is not given explicitly, this bound fails to clarify whether $\nu_{\alpha}$ has any mass near its median. On the other hand, this lower bound on the tails of $\nu_{\alpha}$ ensure that $\nu_{\alpha}$ is not supported on some bounded interval, and hence that strict inequality holds on at least one side of \eqref{desiredinequalities}. In particular, we have proved the first half of Theorem \ref{themaintheorem}.

Next, we study lower bounds for $\delta(P_{\alpha})$. This involves determining upper bounds on the tails of $\nu_{\alpha}$. Our approach uses the two-sided Laplace transform method of Lamzouri \cite{Lamzouri}, which is a simplification of a method of Montgomery \cite[\S 3]{Montgomery2}.

\begin{proposition}[cf.{} {\cite[Proposition 4.1]{Lamzouri}}]
For any $V \geq \mu_{\alpha}$, we have the bound
\begin{equation}{\label{Gaussian}}
\nu_{\alpha}([V,\infty)) \leq \exp\left(- \frac{\left(V - \mu_{\alpha}\right)^2}{2 \sigma^2_{\alpha}}\right).
\end{equation}
\end{proposition}

\begin{proof}
The two-sided Laplace transform of $\nu_{\alpha}$,
	\[\LL_{\alpha}(s) = \int_{\R}{e^{-sx} \, d\nu_{\alpha}(x)},
\]
is equal to $M_{X_{\alpha}}(-s)$, the moment-generating function of the random variable $X_{\alpha}$, with $X_{\alpha}$ as in \eqref{Xalpha}. So by Proposition \ref{probmom} and the fact that $\widetilde{I}_0(x)$ is an even function,
	\[\LL_{\alpha}(s) = \exp\left(- \frac{s}{(1 - 2\alpha) \zeta(1/2)}\right) \prod_{\gamma > 0}{\widetilde{I}_0\left(\frac{2 |\zeta(2\rho)| s}{|\rho - \alpha| |\zeta'(\rho)|}\right)};
\]
we note that the inequality \eqref{Besseldecay2} on $\widetilde{I}_0(x)$ and the fact that $\sigma_{\alpha}^2$ is finite ensures that $\LL_{\alpha}(s)$ is finite for all $s \in \R$. By Chernoff's inequality, we therefore have that for $V \geq \mu_{\alpha}$,
\begin{align*}
\nu_{\alpha}([V,\infty)) & \leq e^{sV} \int_{\R}{e^{-sx} \, d\nu_{\alpha}(x)}	\\
& = \exp\left(s\left(V - \frac{1}{(1 - 2\alpha) \zeta(1/2)}\right)\right) \prod_{\gamma > 0}{\widetilde{I}_0\left(\frac{2 |\zeta(2\rho)| s}{|\rho - \alpha| |\zeta'(\rho)|}\right)}.
\end{align*}
Applying the inequality \eqref{Besseldecay2},
	\[\nu_{\alpha}([V,\infty)) \leq \exp\left(s\left(V - \frac{1}{(1 - 2\alpha) \zeta(1/2)}\right) + s^2 \sum_{\gamma > 0}{\frac{|\zeta(2\rho)|^2}{(|\rho - \alpha| |\zeta'(\rho)|)^2}}\right).
\]
This inequality is minimised by choosing
	\[s = - \left(V - \frac{1}{(1 - 2\alpha) \zeta(1/2)}\right) \left(2 \sum_{\gamma > 0}{\frac{|\zeta(2\rho)|^2}{(|\rho - \alpha| |\zeta'(\rho)|)^2}}\right)^{-1},
\]
which yields the result.
\end{proof}

By taking $V = 0$ in \eqref{Gaussian}, we obtain the following precise lower bound on $\delta(P_{\alpha})$.

\begin{corollary}\label{Palphalowerbound}
For each $0 \leq \alpha < 1/2$,
	\[\delta(P_{\alpha}) \geq 1 - \exp\left(- \frac{\mu_{\alpha}^2}{2 \sigma_{\alpha}^2}\right).
\]
\end{corollary}

This does not quite prove that $\delta(P_{\alpha}) > 1/2$; as we discussed earlier, the infinite sums defining the variances $\sigma_{\alpha}^2$ are not even known to converge unconditionally, let alone have bounds, so we cannot convert the above bound into something more concrete. We must mention, however, that Richard Brent (personal communication) has used the first $65\,536$ zeroes of $\zeta(s)$ to obtain a conjectured value of $\sigma_{\alpha}^2$ when $\alpha = 0$, namely
\begin{equation}\label{variance}
\sigma_0^2 \approx 0.073219.
\end{equation}
Together with Corollary \ref{Palphalowerbound}, this suggests that
	\[\delta(P_0) \geq 0.959321,
\]
so it seems likely that the set of counterexamples to P\'{o}lya's conjecture has small, albeit strictly positive, logarithmic density. Indeed, Brent (personal communication) has performed calculations based on a certain probabilistic heuristic that suggest that
	\[\delta(P_0) \approx 0.99988.
\]

The bound in Corollary \ref{Palphalowerbound} is easily understood in the limit as $\alpha$ tends to $1/2$ from below. Indeed,
	\[\lim_{\alpha \to 1/2^{-}} \sigma_{\alpha}^2 = 2 \sum_{\gamma > 0}{\frac{|\zeta(2\rho)|^2}{(\gamma |\zeta'(\rho)|)^2}},
\]
which is finite and positive under the assumption of $J_{-1}(T) \ll T$, while
	\[\lim_{\alpha \to 1/2^{-}} \mu_{\alpha}^2 = \infty,
\]
and hence we may prove the second half of Theorem \ref{themaintheorem}.

\begin{corollary}
Assume RH, LI, and $J_{-1}(T) \ll T$. Then
	\[\lim_{\alpha \to 1/2^{-}} \delta(P_{\alpha}) = 1.
\]
\end{corollary}

We observe that this result could be proved by using Chebyshev's inequality in place of \eqref{Gaussian}, but the latter gives stronger bounds. On the other hand, using \cite[\S 3 Theorem 1]{Montgomery2} we may show that for every $\e > 0$, there exists an absolute constant $c > 0$ independent of $\alpha$ such that for every $V \gg \mu_{\alpha}$,
\begin{equation}{\label{uppertail}}
\nu_{\alpha}([V,\infty)) \leq \exp\left(-\exp\left(c (V - \mu_{\alpha})^{2/3 - \e}\right)\right),
\end{equation}
but unlike \eqref{Gaussian}, the constant $c$ is not explicitly defined. Finally, if we combine \eqref{uppertail} with \eqref{lowertail}, we see that we have the bounds
	\[\exp\left(-\exp\left(c_1 (V - \mu_{\alpha})^{1 + \e}\right)\right) \leq \nu_{\alpha}([V,\infty)) \leq \exp\left(-\exp\left(c_2 (V - \mu_{\alpha})^{2/3 - \e}\right)\right)
\]
for absolute constants $c_1, c_2 > 0$. If in place of the bound $J_{-1}(T) \ll T$ we assume the stronger bounds
	\[\sum_{0 < \gamma < T}{\frac{|\zeta(2\rho)|}{|\rho - \alpha| |\zeta'(\rho)|}} \asymp (\log T)^{5/4} \qquad \text{and} \qquad \sum_{\gamma > T}{\frac{|\zeta(2\rho)|}{(|\rho - \alpha| |\zeta'(\rho)|)^2}} \asymp \frac{1}{T},
\]
then these bounds can be improved to
	\[\exp\left(-\exp\left(c_1 (V - \mu_{\alpha})^{4/5}\right)\right) \leq \nu_{\alpha}([V,\infty)) \leq \exp\left(-\exp\left(c_2 (V - \mu_{\alpha})^{4/5}\right)\right)
\]
for some absolute constants $c_1,c_2 > 0$; cf.{} \cite[Corollary 12]{Ng}. Based on bounds of this form, Ng gives a heuristic argument in \cite[\S 4.3]{Ng} suggesting that the correct maximal order of growth of $M(x)$ is $\sqrt{x} (\log \log \log x)^{5/4}$ (a conjecture originally put forth by Gonek), and a similar argument suggests that the correct maximal order of growth of $L_{\alpha}(x)$ is $x^{1/2 - \alpha} (\log \log \log x)^{5/4}$.

\section{The $\alpha = 1/2$ Conjecture}

In this section, we prove Theorem \ref{alpha=1/2theorem}, namely that the logarithmic density of $\{x \in [1,\infty) : L_{1/2}(x) \leq 0\}$ is equal to $1$ assuming the Riemann hypothesis and $J_{-1}(T) \ll T$. We begin by determining an explicit expression for $L_{1/2}(x)$ in terms of a sum over the zeroes of $\zeta(s)$. In the range $0 \leq \alpha < 1/2$, an explicit expression for $L_{\alpha}(x)$ was found in Theorem \ref{ThmL(x)alongT_v}. A simple modification of the proof of this theorem shows that a similar result holds for $\alpha = 1/2$.

\begin{theorem}[cf.{} {\cite[Equation (7)]{Mossinghoff}}]
Assume RH and that all of the zeroes of $\zeta(s)$ are simple. Then for $T_v \in \TT$ and $x \geq 1$, we have that
\begin{equation}{\label{L{1/2}(x)alongT_v}}
\begin{split}
L_{1/2}(x) & = \frac{\log x}{2 \zeta(1/2)} + \frac{\gamma_0}{\zeta(1/2)} - \frac{\zeta'(1/2)}{\zeta(1/2)^2} + \sum_{|\gamma| < T_v}{\frac{\zeta(2\rho)}{\zeta'(\rho)} \frac{x^{i \gamma}}{i \gamma}}	\\
& \qquad + E_{1/2}(x) + I_{1/2}(x) + R_{1/2}(x,T_v),
\end{split}
\end{equation}
where $E_{1/2}(x) = \lambda(x)/(2\sqrt{x})$ if $x \in \N$ and $0$ otherwise, and for arbitrary small $0 < \e < 1/2$,
	\[I_{1/2}(x) = \frac{1}{2\pi i \sqrt{x}} \int^{\e + i\infty}_{\e - i\infty}{\frac{\zeta(2s)}{\zeta(s)} \frac{x^s}{s - 1/2} \, ds},
\]
and
	\[R_{1/2}(x,T_v) \ll \frac{1}{\sqrt{x}} + \frac{\sqrt{x} \log x}{T_v} + \frac{\sqrt{x}}{T_v^{1-\e} \log x}
\]
with the implied constant dependent on $\e$. Moreover, $\lim_{v \to \infty} R_{1/2}(x,T_v) = 0$.
\end{theorem}

The key difference is that we have a double pole at $s = 1/2$ with
	\[\Res_{s = 1/2} \frac{\zeta(2s)}{\zeta(s)} \frac{x^{s - 1/2}}{s - 1/2} = \frac{\log x}{2 \zeta(1/2)} + \frac{\gamma_0}{\zeta(1/2)} - \frac{\zeta'(1/2)}{\zeta(1/2)^2},
\]
which is easily verified by the fact that
	\[\zeta(s) = \frac{1}{s - 1} + \gamma_0 + O(s - 1)
\]
as $s$ tends to $1$ \cite[Corollary 1.16]{Montgomery}.

Next, we must show that $I_{1/2}(x) = O(1)$. This is slightly different to Lemma \ref{Fawazalphaboundlemma}, as the bound \eqref{Fawazbound} is inadequate to show the convergence of
	\[\int^{x}_{0}{\frac{I(u)}{u^{3/2}} \, du}.
\]
However, the bound \eqref{Fawazbound} does suffice to show that
	\[\int^{x}_{1}{\frac{I(u)}{u^{3/2}} \, du} = \frac{1}{2\pi i} \int^{\e + i\infty}_{\e - i\infty}{\frac{\zeta(2s)}{\zeta(s)} \frac{x^{s - 1/2}}{s(s - 1/2)} \, ds} - \frac{1}{2\pi i} \int^{\e + i\infty}_{\e - i\infty}{\frac{\zeta(2s)}{\zeta(s)} \frac{1}{s(s - 1/2)} \, ds},
\]
for all $x \geq 1$, and hence that for $x \geq 1$,
	\[I_{1/2}(x) = \frac{I(x)}{\sqrt{x}} + \frac{1}{2} \int^{x}_{1}{\frac{I(u)}{u^{3/2}} \, du} + \frac{1}{4\pi i} \int^{\e + i\infty}_{\e - i\infty}{\frac{\zeta(2s)}{\zeta(s)} \frac{1}{s(s - 1/2)} \, ds},
\]
and this is $\ll 1$ as $x$ tends to infinity. This allows us to mimic the proof of Corollary \ref{CorL(x)alongT} in order to obtain the following analogous result.

\begin{corollary}{\label{CorL{1/2}(x)arbitraryT}}
Assume RH and $J_{-1}(T) \ll T$. Then for all $T \geq 1$ and $x \geq 1$, we have that
	\[L_{1/2}(x) = \frac{\log x}{2 \zeta(1/2)} + \sum_{|\gamma| < T}{\frac{\zeta(2\rho)}{\zeta'(\rho)} \frac{x^{i \gamma}}{i \gamma}} + R_{1/2}(x,T),
\]
where for arbitrary small $0 < \e < 1/2$,
	\[R_{1/2}(x,T) \ll 1 + \frac{\sqrt{x} \log x}{T} + \frac{\sqrt{x}}{T^{1-\e} \log x}.
\]
\end{corollary}

Note that the constant terms in \eqref{L{1/2}(x)alongT_v} have now been absorbed by the error term $R_{1/2}(x,T)$.

Finally, we are able to prove Theorem \ref{alpha=1/2theorem}. It suffices to prove the following proposition; Theorem \ref{alpha=1/2theorem} then follows as $\log x / (2 \zeta(1/2))$ is negative and grows faster than $(\log \log x)^{3/2} \log \log \log x$. 

\begin{proposition}[cf.{} {\cite[Theorem 1 (ii)]{Ng}}]
Assume RH and $J_{-1}(T) \ll T$. Then for some sufficiently large $\beta > 0$, the set
	\[S_{1/2} = \left\{x \in [e^9,\infty) : \left|L_{1/2}(x) - \frac{\log x}{2 \zeta(1/2)}\right| \geq \beta (\log \log x)^{3/2} \log \log \log x\right\}
\]
has logarithmic density zero.
\end{proposition}

\begin{proof}
By Corollary \ref{CorL{1/2}(x)arbitraryT}, for $x \geq 1$ and $T \geq 1$,
\begin{align*}
L_{1/2}(x) & = \frac{\log x}{2 \zeta(1/2)} + \sum_{|\gamma| < (\log T)^4}{\frac{\zeta(2\rho)}{\zeta'(\rho)} \frac{x^{i \gamma}}{i \gamma}} + \sum_{(\log T)^4 \leq |\gamma| < T}{\frac{\zeta(2\rho)}{\zeta'(\rho)} \frac{x^{i \gamma}}{i \gamma}}	\\
& \qquad + O\left(1 + \frac{\sqrt{x} \log x}{T} + \frac{\sqrt{x}}{T^{1-\e} \log x}\right)
\end{align*}
for arbitrary small $\e > 0$. By \eqref{J_{-1/2}(T)},
	\[\sum_{|\gamma| < (\log T)^4}{\frac{\zeta(2\rho)}{\zeta'(\rho)} \frac{x^{i \gamma}}{i \gamma}} \ll \sum_{0 < \gamma < (\log T)^4}{\frac{|\zeta(2\rho)|}{\gamma |\zeta'(\rho)|}} \ll (\log \log T)^{3/2} \log \log \log T.
\]
Thus if we restrict ourselves to the range $T \leq x \leq eT$,
	\[L_{1/2}(x) = \frac{\log x}{2 \zeta(1/2)} + \sum_{(\log T)^4 \leq |\gamma| < T}{\frac{\zeta(2\rho)}{\zeta'(\rho)} \frac{x^{i \gamma}}{i \gamma}} + O\left((\log \log x)^{3/2} \log \log \log x\right).
\]
Let $C > 0$ be the implicit constant in the error term above. Then if $x \in S_{1/2} \cap [T, eT]$,
\begin{align*}
\left|\sum_{(\log T)^4 \leq |\gamma| < T}{\frac{\zeta(2\rho)}{\zeta'(\rho)} \frac{x^{i \gamma}}{i \gamma}}\right| & \geq \left|L_{1/2}(x) - \frac{\log x}{2 \zeta(1/2)}\right| - C(\log \log x)^{3/2} \log \log \log x	\\
& \geq (\beta - C) (\log \log x)^{3/2} \log \log \log x
\end{align*}
and hence if $\beta > C - 1$, then
	\[\left|\sum_{(\log T)^4 \leq |\gamma| < T}{\frac{\zeta(2\rho)}{\zeta'(\rho)} \frac{x^{i \gamma}}{i \gamma}}\right| \geq (\log \log x)^{3/2} \log \log \log x
\]
for $x \in S_{1/2} \cap [T, eT]$. By taking $T = e^k$ for any $k \geq 9$, we see that if $x \in S_{1/2} \cap [e^k, e^{k + 1}]$, then
	\[\frac{1}{(\log k)^3 (\log \log k)^2} \left|\sum_{k^4 \leq |\gamma| < e^k}{\frac{\zeta(2\rho)}{\zeta'(\rho)} \frac{x^{i \gamma}}{i \gamma}}\right|^2 \geq 1.
\]
Thus for any $X \geq e^9$,
\begin{align*}
\int_{S_{1/2} \cap [1,X]}{\frac{dx}{x}} & \leq \sum^{\lfloor \log X \rfloor}_{k = 9} \int_{S_{1/2} \cap [e^k,e^{k + 1}]}{\frac{dx}{x}}	\\
& \leq \sum^{\lfloor \log X \rfloor}_{k = 9} \frac{1}{(\log k)^3 (\log \log k)^2} \int^{e^{k + 1}}_{e^k}{\left|\sum_{k^4 \leq |\gamma| < e^k}{\frac{\zeta(2\rho)}{\zeta'(\rho)} \frac{x^{i \gamma}}{i \gamma}}\right|^2 \, \frac{dx}{x}}.
\end{align*}
We may bound these integrals via Lemma \ref{keyboundlemma} with $T = k^4$, $X = Z = e^k$. Thus
	\[\int_{S_{1/2} \cap [1,X]}{\frac{dx}{x}} \ll \sum^{\lfloor \log X \rfloor}_{k = 9} \frac{1}{k (\log k)^2} \ll 1,
\]
and hence
	\[\delta(S_{1/2}) = \lim_{X \to \infty} \frac{1}{\log X} \int_{S_{1/2} \cap [1,X]}{\frac{dx}{x}} = 0.
\qedhere\]
\end{proof}

Finally, we mention that a modification of the heuristic argument of Ng that suggests that the correct order of growth of $M(x)$ is $\sqrt{x} (\log \log \log x)^{5/4}$ yields a similar heuristic for the order of growth of
	\[L_{1/2}(x) - \frac{\log x}{2 \zeta(1/2)}.
\]
This leads us to suggest the following refinement of the $\alpha = 1/2$ conjecture.

\begin{conjecture}\label{L1/2conj}
As $x$ tends to infinity,
	\[L_{1/2}(x) \sim \frac{\log x}{2 \zeta(1/2)}.
\]
\end{conjecture}

This conjectural asymptotic was first put forth by Wintner \cite[Equation (2)]{Wintner}, who observed that the hypothesis $L(x) = O\left(\sqrt{x}\right)$ implies Conjecture \ref{L1/2conj}; in light of Theorem \ref{Inghamthm}, however, this hypothesis seems most likely false. The remarkable observation is that Ng's heuristic argument on the correct order of growth of $M(x)$ suggests that Conjecture \ref{L1/2conj} is nevertheless true.

\subsection*{Acknowledgements}

The author is greatly indebted Richard Brent for first suggesting this approach to P\'{o}lya's conjecture as an undergraduate project, and for the many helpful discussions that followed. Many thanks also to Tim Trudgian for his suggestion of possible extension of the author's results, and for the feedback that he provided. The author would also like to thank the anonymous referee for the many corrections and helpful suggestions.

\end{document}